\documentclass{amsart}

\usepackage{amssymb, amsmath}
\usepackage{mathrsfs}
\usepackage{amscd}
\usepackage{verbatim}

\usepackage{enumerate}

\usepackage[colorlinks,linkcolor={blue},citecolor={blue},urlcolor={red},]{hyperref}

\usepackage[colorinlistoftodos,prependcaption,textsize=tiny]{todonotes}

\theoremstyle{plain}
\newtheorem{theorem}{Theorem}[section]
\theoremstyle{remark}
\newtheorem{remark}[theorem]{Remark}

\theoremstyle{plain}
\newtheorem{corollary}[theorem]{Corollary}
\newtheorem{lemma}[theorem]{Lemma}
\newtheorem{proposition}[theorem]{Proposition}
\newtheorem{definition}[theorem]{Definition}

\newtheorem{assumption}[theorem]{Assumption}

\numberwithin{equation}{section}

\def\N{{\mathbb N}}

\def\R{{\mathbb R}}
\def\C{{\mathbb C}}

\newcommand{\Di}{\mathcal{D}}
\newcommand{\E}{{\mathbb E}}
\renewcommand{\P}{{\mathbb P}}
\newcommand{\F}{{\mathscr F}}
\newcommand{\A}{{\mathscr A}}
\newcommand{\G}{{\mathscr G}}
\renewcommand{\H}{{\mathscr H}}

\newcommand{\g}{\gamma}

\newcommand{\om}{\omega}
\renewcommand{\O}{\Omega}

\newcommand{\Ito}{{\hbox{\rm It\^o}}}

\newcommand{\angH}{\omega_{H^{\infty}}}

\newcommand{\tr}{\mathrm{tr}}

\renewcommand{\Re}{\hbox{\rm Re}\,}

\newcommand{\calL}{{\mathscr L}}
\newcommand{\Sc}{{\mathcal S}}

\newcommand{\one}{{{\bf 1}}}

\newcommand{\lb}{\langle}
\newcommand{\rb}{\rangle}

\newcommand{\wh}{\widehat}
\newcommand{\supp}{\text{\rm supp\,}}

\newcommand{\RegJ}{\text{\normalfont{SMR}}(p,T)}
\newcommand{\RegJw}{\text{\normalfont{SMR}}^0(p,T)}

\newcommand{\RegRw}{\text{\normalfont{SMR}}^0(p,\infty)}
\newcommand{\RegR}{\text{\normalfont{SMR}}(p,\infty)}
\newcommand{\RegRtwoo}{\text{\normalfont{SMR}}(2,\infty)}

\newcommand{\RegJtwointerval}{\text{\normalfont{SMR}}(p,T_2)}
\newcommand{\RegJone}{\text{\normalfont{SMR}}(p,T_1)}
\newcommand{\RegJlambda}{\text{\normalfont{SMR}}(p, T/\lambda)}

\newcommand{\RegJalpha}{\text{\normalfont{SMR}}(p,T,\alpha)}
\newcommand{\RegRalpha}{\text{\normalfont{SMR}}(p,\infty,\alpha)}

\newcommand{\RegthetaR}{\text{\normalfont{SMR}}_{\theta}(p,\infty)}

\newcommand{\RegpsiR}{\text{\normalfont{SMR}}_{\psi}(p,\infty)}
\newcommand{\MaxDetR}{\text{\normalfont{DMR}}(p,\infty)}
\newcommand{\MaxDetJ}{\text{\normalfont{DMR}}(p,T)}

\newcommand{\ud}{d}
\newcommand{\wt}{\widetilde}
\renewcommand{\tilde}{\widetilde}

\newcommand{\BIP}{\text{\normalfont{BIP}}}
\newcommand{\sign}{{\rm sign}}

\allowdisplaybreaks

\begin{document}

\author{Antonio Agresti}
\address{Department of Mathematics Guido Castelnuovo\\ Sapienza University of Rome\\ P.le
A. Moro 2\\ 00100 Roma\\ Italy.} \email{agresti@mat.uniroma1.it}

\author{Mark Veraar}
\address{Delft Institute of Applied Mathematics\\
Delft University of Technology \\ P.O. Box 5031\\ 2600 GA Delft\\The
Netherlands.} \email{M.C.Veraar@tudelft.nl}

\thanks{The second author is supported by the VIDI subsidy 639.032.427 of the Netherlands Organisation for Scientific Research (NWO)}

\date\today

\title{Stability properties of stochastic maximal $L^p$-regularity}



\begin{abstract}
In this paper we consider $L^p$-regularity estimates for solutions to stochastic evolution equations, which is called stochastic maximal $L^p$-regularity. Our aim is to find a theory which is analogously to Dore's theory for deterministic evolution equations. He has shown that maximal $L^p$-regularity is independent of the length of the time interval, implies analyticity and exponential stability of the semigroup, is stable under perturbation and many more properties. We show that the stochastic versions of these results hold.
\end{abstract}

\maketitle
\setcounter{tocdepth}{1}
\tableofcontents

\section{Introduction}

In this paper we study sharp $L^p$-regularity estimates for solutions to stochastic evolution equations. This we will call stochastic maximal $L^p$-regularity. From a PDE point of view it leads to natural a priori estimates, and this can in turn be used to obtain local existence and uniqueness for nonlinear PDEs (see e.g.\ \cite{PrussWeight1,pruss2016moving,CriticalQuasilinear}). In the deterministic setting \cite{Dore} Dore has found several stability properties of maximal $L^p$-regularity (see also the monograph \cite{pruss2016moving}). A list of results can be found below Definition \ref{def:MRLp}. These properties are interesting to know from a theoretical point of view. In practice one usually checks the conditions of Weis' theorem which states that maximal $L^p$-regularity is equivalent to $R$-sectoriality if the underlying space is a UMD space. If $p=1$, $p=\infty$ or $X$ is not UMD, then one can not rely on the latter results, and thus Dore's theory becomes more relevant. Alternative ways to derive maximal $L^p$-regularity can be to use the Da Prato-Grisvard theorem (see \cite[Theorem 9.3.5]{Haase:2}) or put more restrictive conditions on the generator $A$ (see \cite{Kalton-Kucherenko}).

In \cite{MaximalLpregularity,NVW11,NVW13} stochastic maximal $L^p$-regularity for an operator $A$ (or briefly $A	\in\RegJ$) was proved under the condition that $A$ has a bounded $H^\infty$-calculus (see Theorem \ref{thm:SMRmain} below). These results have been applied in several other papers (e.g.\ \cite{A18, Hornung, NVW11eq}). Recently, extensions to the time and $\Omega$-dependent setting have been obtained in \cite{VP18}. The stochastic maximal regularity theory of the above mentioned papers provides an alternative approach and extension of a part of Krylov's $L^p$-theory for stochastic PDEs (see \cite{Kry} and the overview \cite{KryOverview}).

The aim of the first part of the current paper is to obtain stochastic versions of Dore's results \cite{Dore}. In many cases completely new proofs are required due to the fact that stochastic convolutions behave in very different way. Assume $-A$ generates a strongly continuous semigroup $(S(t))_{t\geq 0}$ on a Banach space $X$ with UMD and type $2$. In Sections \ref{s:SMR}--\ref{s:weight}, for all $p\in [2, \infty]$ and $T\in (0,\infty]$, we obtain the following stability properties of stochastic maximal $L^p$-regularity:
\begin{itemize}
\item the class $\RegJ$ is stable under appropriate translations and dilations;
\item independence of the dimension of the noise;
\item if $A\in \RegJ$, then $S$ is an analytic semigroup;
\item if $A\in \RegR$, then $S$ is exponentially stable;
\item $\RegR\subseteq \RegJ = \text{\normalfont{SMR}}(p,\tilde{T})$, for any $\tilde{T}\in (0,\infty)$.
\item if $A\in\RegJ $ and $S$ is exponentially stable, then $A\in\RegR$;
\item perturbation results;
\item weighted characterizations.
\end{itemize}
A $p$-independence result similar to Dore's result holds as well, but it is out of the scope of this paper to prove this. Note that in \cite{Dore} the $p$-independence in the deterministic case was derived from operator-valued Calder\'on--Zygmund theory. A stochastic Calder\'on--Zygmund theory has been recently obtained in \cite{LoVer} where among other things the $p$-independence of $\RegJ$ is established.

The aim of the second part of the paper is to introduce a weighted version of stochastic maximal regularity (see Section \ref{s:weight}). In a future paper we will use the theory of the current paper to study quasilinear stochastic evolution equations. In particular we plan to obtain a version of \cite{HornungDissertation,Hornung} with weights in time. Because of the weights in time one can treat rough initial data. This has already been demonstrated by Portal and the second author in \cite{VP18} in the semilinear case.

\subsubsection*{Notation}
We write $A \lesssim_P B$, whenever there is a constant $C$ only depending on the parameter $P$ such that $A\leq C B$. Moreover, we write $A \eqsim_P B$ if $A \lesssim_P B$ and $A \gtrsim_P B$.

\subsubsection*{Acknowledgment} The authors would like to thank Emiel Lorist for helpful comments. The authors would also like to thank Bounit Hamid for pointing out the reference \cite{BounitDrEl} for Lemma \ref{lem:analyticLp}.

\section{Preliminaries}
In this section we collect some useful facts and fix the notation, which will be employed through the paper.

\subsection{Sectorial Operators and $H^{\infty}$-calculus}
For details on the $H^\infty$-calculus we refer the reader to \cite{Haase:2,Analysis2,KuWe}. For $\varphi\in (0,\pi)$ we denote by $\Sigma_{\varphi}:=\{z\in \mathbb{C}\,:\,|\arg (z)|<\varphi\}$ the open sector of angle $\varphi$. Moreover, for a closed linear operator $A$ on a Banach space $X$, $D(A)$ and $R(A)$ denote its domain and range respectively. We say that $A$ is sectorial if $A$ is injective, $\overline{R(A)}=\overline{D(A)}=X$ and there exists $\varphi\in (0,\pi)$ such that $\sigma(A)\subseteq \Sigma_{\varphi}$ and
\begin{equation*}
\sup_{z\in \C \setminus \overline{\Sigma_{\varphi}}}\|zR(z,A)\|_{\calL(X)}<\infty.
\end{equation*}
Moreover, we $\om(A)$ denotes the infimum of all $\varphi\in (0,\pi)$ such that $A$ is sectorial of angle $\varphi$.

For $\varphi\in(0,\pi)$, we denote by $H^{\infty}_0(\Sigma_{\varphi})$ the set of all holomorphic function $f:\Sigma_{\varphi}\rightarrow \C$ such that $|f(z)|\leq C|z|^{\varepsilon}/(1+|z|^{2\varepsilon})$ for some $C,\varepsilon>0$ independent of $z\in \Sigma_{\varphi}$. Let $A$ be a sectorial operator of angle $\om(A)<\nu<\varphi$. Then for $f\in H^{\infty}_0(\Sigma_{\varphi})$ we set
\begin{equation}
\label{eq:dunford}
f(A):=\frac{1}{2\pi i} \int_{\partial\Sigma_{\nu}} f(z)R(z,A)\,dz;
\end{equation}
where the orientation of $\partial \Sigma_{\nu}$ is such that $\sigma(A)$ is on the right. By \cite[Section 10.2]{Analysis2}, $f(A)$ is well-defined in $\calL(X)$ and it is independent of $\nu\in (\om(A),\varphi)$.

Furthermore, the operator $A$ is said to have a bounded $H^{\infty}(\Sigma_{\varphi})$-calculus if there exists $C>0$ such that for all $f\in H^{\infty}_0(\Sigma_{\varphi})$,
\begin{equation*}
\|f(A)\|_{\calL(X)}\leq C\|f\|_{H^{\infty}(\Sigma_{\varphi})}\,,
\end{equation*}
where $\|f\|_{H^{\infty}(\Sigma_{\varphi})}=\sup_{z\in\Sigma_{\varphi}}|f(z)|$. Lastly, $\angH(A)$ denotes the infimum of all $\varphi\in (\om(A),\pi)$ such that $A$ has a bounded $H^{\infty}(\Sigma_{\varphi})$-calculus.

\begin{remark}
Nowadays it is known that a large class of elliptic operators have a bounded $H^{\infty}$-calculus. For instances see \cite{DDHPV}, \cite[Example 3.2]{NVW11eq}, \cite[Subsection 1.3]{VP18}, \cite[Section 10.8]{Analysis2} and in the reference therein.
\end{remark}

Let $\BIP(X)$ denote the set of sectorial operators which have bounded imaginary powers, i.e. $A^{it}$ extends to a bounded linear operator on $X$ and $\sup_{|t|\leq 1}\|A^{it}\|_{\calL(X)}<\infty$. Moreover, we set
$$
\theta_A :=\limsup_{|t|\to \infty} \frac{1}{|t|}\log \|A^{it}\|_{\calL(X)}.
$$
If $A$ has a bounded $H^{\infty}(\Sigma_{\nu})$-calculus for some $\nu\in(0,\pi)$, then $A\in\BIP(X)$ and $\theta_A\leq \om_{H^{\infty}}(A)$.

Let $(r_n)_{n\geq 1}$ be a Rademacher sequence on $(\Omega,\mathcal{F},\P)$, i.e. a sequence of independent random variables with $\P(r_n = 1) = \P(r_n = -1) = \frac12$ for all $n\geq 1$. A family of bounded linear operators $\mathcal{T}\subseteq \calL(X,Y)$ is said to be {\em $R$-bounded} if there exists a constant $C>0$ such that for all $x_1,\dots,x_N\in X$, $T_1,\dots,T_N\in\mathcal{J}$ one has
$$
\Big\|\sum_{j=1}^N r_j T_j x_j\Big\|_{L^2(\Omega;X)} \leq C\Big\|\sum_{j=1}^N r_j  x_j\Big\|_{L^2(\Omega;X)}.
$$
For more on this notion see \cite[Chapter 8]{Analysis2}.

An operator $A$ is called {\em $R$-sectorial} if for some $\sigma\in (0,\pi)$ one has $\C\setminus \Sigma_{\sigma} \subseteq \rho(A)$ and the set $\{\lambda R(\lambda,A)\,:\,\lambda\in \C\setminus \Sigma_{\sigma}\}$ is $R$-bounded. Finally, $\om_{R}(A)$ denotes the infimum of such $\sigma$'s. For more on this see \cite{Analysis2,pruss2016moving}.

\begin{remark}
\label{r:BIPimplies}
Let $X$ be a UMD Banach space. Then $A\in \BIP(X)$ implies that $A$ is $R$-sectorial on $X$ and $\om_{R}(A)\leq \theta_A$ (see \cite[Theorem 4.4.5]{pruss2016moving}).
\end{remark}
For details on UMD spaces we refer to \cite[Chapter 4]{Analysis1}.

\subsection{Deterministic Maximal $L^p$-regularity and R-boundedness}
\label{ss:DMR}
Deterministic maximal $L^p$-regularity has been investigated by many authors and plays an important role in the modern treatment of parabolic equations, see e.g. \cite{DHP,KuWe,pruss2016moving,CriticalQuasilinear} and the references therein.

If $-A$ generates a strongly continuous semigroup $S:=(S(t))_{t\geq 0}$, then $\omega_0(-A)$ denotes the exponential growth bound of $S$
\[\omega_0(-A):=\inf\{\omega\in \R\,:\,\sup_{t>0}e^{-\om t}\|S(t)\|<\infty\}.\]
Thus $\omega_0(-A)<0$ if and only if $S$ is exponentially stable. Moreover, if $A$ is a densely defined operator and $w>\omega_0(-A)$, then $w+A$ is a sectorial operator on $X$; thus one can define $(w+A)^{1/2}$ as a closed operator on $X$.

\begin{definition}[Deterministic maximal $L^p$-regularity]\label{def:MRLp}
Let $T>0$ and $p\in [1\,\infty]$. A closed linear operator $A$ on a Banach space $X$ is said to have {\em (deterministic) maximal $L^p$-regularity on $(0,T)$} if for all $f\in L^p(0,T;X)$ there exists an unique $u\in W^{1,p}(0,T;X)\cap L^p(0,T;D(A))$ such that
$$
u'+Au=f, \qquad u(0)=0.
$$
In this case we write $A\in \MaxDetJ$.
\end{definition}
Stability properties of the deterministic maximal $L^p$-regularity have been studied in \cite{Dore} (see also the monograph \cite{pruss2016moving}): For all $p\in [1, \infty]$ and $T\in (0,\infty]$
\begin{itemize}
\item the class $\MaxDetJ$ is stable under appropriate translations and dilations;
\item if $A\in \MaxDetJ$, then $-A$ generates an analytic semigroup;
\item if $A\in \MaxDetR$, then $\omega_0(-A)<0$;
\item $\MaxDetR\subseteq \MaxDetJ = \text{\normalfont{DMR}}(p,\tilde{T})$ if $T,\tilde{T}\in (0,\infty)$.
\item if $A\in\MaxDetJ $ and $\omega_0(-A)<0$, then $A\in\MaxDetR$;
\item perturbation results;
\item $\MaxDetJ \subseteq \text{\normalfont{DMR}}(q,T)$ for all $q\in (1, \infty)$ with equality if $p\in (1, \infty)$.
\end{itemize}
Finally let us mention that weighted versions of deterministic maximal $L^p$-regularity have been studied in \cite{PruSim04} for power weights and in \cite{ChillFio,ChillKrol} for weights of $A_p$-type.

The following result was proven in \cite{We}, it has been very influential and is by now a classical result: for a UMD space $X$, $p\in (1, \infty)$ and $0\in \rho(A)$ one has $A\in \MaxDetR$ if and only if $A$ is $R$-sectorial of angle $<\pi/2$.

\subsection{$\gamma$-radonifying operators}
\label{ss:radonifyingop}
In this subsection we briefly review some basic facts regarding $\gamma$-radonifying operators; for further discussions see \cite[Chapter 9]{Analysis2}. Through this subsection $(\g_n)_{n\in \N}$ denotes a Gaussian sequence, i.e. a sequence of independent standard normal variables over a probability space $(\tilde{\Omega},\tilde{\mathcal{A}},\tilde{\P})$.

Let $\H$ be a Hilbert space (with scalar product $[\cdot,\cdot]$) and $X$ be a Banach space with finite cotype. Recall that $H\otimes X$ is the space of finite rank operators from $\H$ to $X$. In other words, each $T\in H \otimes X$ has the form
\begin{equation*}
T=\sum_{n=1}^N h_n \otimes x_n\,,
\end{equation*}
for $N\in \N$ and $(h_n)_{n=1}^N\subset\H$. Here $h\otimes x$ denotes the operator $g\mapsto [g,h] x$.

For $T\in H\otimes X$ define
\begin{equation*}
\|T\|_{\g(\H,X)}^2:=\sup \tilde{\E} \Big\|\sum_{n=1}^N \g_n Th_n\Big\|^2_X<\infty\,;
\end{equation*}
where the supremum is taken over all finite orthonormal systems $(h_n)_{n=1}^N$ in $\H$. Then $\|T\| \leq \|T\|_{\gamma(\H,X)}$. The closure of $H\otimes X$ with respect to the above norm is called {\em the space of $\g$-radonifying operators} and  is denoted by $\gamma(\H,X)$.

The following property will be used through the paper.
\begin{proposition}[Ideal Property]
\label{prop:idealprop}
Let $T\in \g(\H,X)$. If $G$ is another Hilbert space and $Y$ a Banach space, then for all $U\in \calL(X,Y)$ and $V\in \calL(G,\H)$ we have $UTV\in \g(G,Y)$ and
\begin{equation*}
\|UTV\|_{\g(G,Y)}\leq \|U\|_{\calL(X,Y)}\|T\|_{\g(\H,X)}\|V\|_{\calL(G,\H)}.
\end{equation*}
\end{proposition}
We will be mainly interested in the case that $\H=L^2(S;H)$ where $(S,\mathcal{A},\mu)$ is a measure space and $H$ is another Hilbert space. In this situation we employ the following notation:
$$\g(S;H,X):=\gamma(L^2(S;H),X)$$
and $\g(a,b;H,X):=\g(L^2(a,b;H),X)$, if $S=(a,b)$, $\mu$ is the one dimensional Lebesgue measure and $\mathcal{A}$ is the natural $\sigma$-algebra. If $H=\R$ we simply write $\g(a,b;X):=\g(L^2(a,b),X)$.

An $H$-strongly measurable function $G:S\rightarrow \calL(H,X)$ (i.e. for each $h\in H$ the map $s\mapsto f(s)h$ is strongly measurable) \textit{belongs to} $L^2(S;H)$ \textit{scalarly} if $G^*(s)x^*\in L^2(S;H)$ for each $x^*\in X^*$. Such a function \textit{represent} an operator $R\in \g(S;H,X)$ if for all $f\in L^2(S;H)$ and $x^*\in X^*$ we have
\begin{equation*}
\int_{S}\langle G(s)f(s),x^*\rangle\,ds= \langle R(f),x^*\rangle.
\end{equation*}
It can be shown that if $R$ is represented by $G_1$ and $G_2$ then $G_1=G_2$ almost everywhere. It will be convenient to identify $R$ with $G$ and we will simply write $G\in \g(S;H,X)$ and $\|G\|_{\g(S;H,X)}:=\|R\|_{\g(S;H,X)}$.
By the ideal property, if $S = S_1\cup S_2$ and $S_1$ and $S_2$ are disjoint, then
\begin{align}\label{eq:Geq}
\|G\|_{\g(S;H,X)}\leq \|G\|_{\g(S_1;H,X)}+\|G\|_{\g(S_2;H,X)}.
\end{align}

Another consequence of the ideal property is that for $G\in \g(S;H,X)$, $\phi\in L^\infty(S)$ and $S_0\subseteq S$, we have
\begin{equation}\label{eq:phiGinfty}
\|\phi G\|_{\g(S;H,X)} \leq \|\phi\|_{\infty} \|G\|_{\g(S;H,X)},  \ \ \ \|\one_{S_0} G\|_{\gamma(S;H,X)} = \|G\|_{\gamma(S_0;H,X)}
\end{equation}

To conclude this section, we recall the following embedding:
\begin{proposition}
\label{prop:inclusiongammatype}
Let $X$ be a Banach space with type 2, then
\begin{equation*}
L^2(S;\g(H,X))\hookrightarrow \g(L^2(S),\gamma(H,X)) \hookrightarrow\g(L^2(S;H),X).
\end{equation*}
\end{proposition}
\begin{proof}
Since $X$ has type $2$, also $\gamma(H,X)$ has type $2$, because it is isomorphic to a closed subspace of $L^2(\tilde{\Omega};X)$ (see \cite[Proposition 7.1.4]{Analysis2}). Now the first embedding follows from \cite[Theorem 9.2.10]{Analysis2}. The second embedding follows by considering finite rank operators and applying \cite[Theorem 7.1.20]{Analysis2} with orthonormal family $\{\tilde{\gamma}_{i} \wh{\gamma}_{j}:i, j\in \N\}$, where $\tilde{\gamma}_i$ and $\wh{\gamma}_j$ are defined on probability spaces $\tilde{\Omega}$ and $\wh{\Omega}$, respectively.
\end{proof}

\subsection{Stochastic Integration in UMD Banach spaces}
The aim of this section is to present basic results of the stochastic integration theory in UMD Banach spaces developed in \cite{NVW1}.
Let $(\O,\A,\P)$ be a probability space with filtration $\F = (\F_t)_{t\geq 0}$ and throughout the rest of the paper it is fixed.
An $\F$-adapted step process is a linear combination of functions
\begin{equation*}
(\one_{A\times (s,t]}\otimes( h\otimes x))(\omega,t):=\one_{A\times (s,t]}(\omega,t)( h\otimes x)\,,
\end{equation*}
where $0\leq s<t\leq T$ and $A\in \F_s$.
Let $T>0$, we say that a stochastic process $G:[0,T]\times\Omega \rightarrow \calL(H,X)$ \textit{belongs to} $L^2(0,T;H)$ \textit{scalarly almost surely} if for all $x^*\in X^*$ a.s.\ the $G^*x^*\in L^2(0,T;H)$. Such a process $G$ is said to represent an $L^2(0,T;H)$-strongly measurable $R\in L^0(\Omega;\g(0,T;H,X))$ if for all $f\in L^2(0,T;H)$ and $x^*\in X^*$ we have
\begin{equation*}
\langle R(\om)f,x^*\rangle= \int_0^T \langle G(t,\om)f(t),x^* \rangle\,dt.
\end{equation*}
As done in Subsection \ref{ss:radonifyingop}, we identify $G$ and $R$ in the case that $R$ is represented by $G$. Moreover, we say that $G\in L^p(\Omega;\g(0,T;H,X))$ if $R\in L^p(\Omega;\g(0,T;H,X))$ for some $p\in[0,\infty)$.
We say that $R:\Omega \rightarrow\g(0,T;H,X)$ is {\em elementary adapted to $\F$} if it is represented by an $\F$-adapted step process $G$. Lastly,
\begin{equation*}
L_{\F}^p(\Omega;\g(0,T;H,X))
\end{equation*}
denotes the closure of all elementary adapted $R\in L^p(\Omega;\g(0,T;H,X))$.
In the paper we will consider cylindrical Gaussian noise.

\begin{definition}
A bounded linear operator $W_H:L^2(\R_+;H)\rightarrow L^2(\Omega)$ is said to be an {\em $\F$-cylindrical Brownian motion} in $H$ if the following are satisfied:
\begin{itemize}
\item for all $f\in L^2(\R_+;H)$ the random variable $W_H(f)$ is centered Gaussian.
\item for all $t\in \R_+$ and $f\in L^2(\R_+;H)$ with support in $[0,t]$, $W_H(f)$ is $\F_t$-measurable.
\item for all $t\in \R_+$ and $f\in L^2(\R_+;H)$ with support in $[t,\infty]$, $W_H(f)$ is independent of $\F_t$.
\item for all $f_1,f_2\in L^2(\R_+;H)$ we have $\E(W_H(f_1)W_H(f_2))=[f_1,f_2]_{L^2(\R_+;H)}$.
\end{itemize}
\end{definition}
Given an $\F$-cylindrical Brownian motion in $H$, the process $(W_H(t)h)_{t\geq 0}$, where
\begin{equation}
W_H(t)h:=W_H(\one_{(0,t]}\otimes h)\,,
\end{equation}
is an $\F$-Brownian motion.

At this point, we can define the \textit{stochastic integral with respect to an $\F$-cylindrical Brownian} motion in $H$ of  the process $\one_{A\times (s,t]}\otimes( h\otimes x)$:
\begin{equation}
\int_0^{\infty} \one_{A\times (s,t]}\otimes( h\otimes x)(s) \,dW_H(s):= \one_{A}\otimes (W_H(t)h-W_H(s)h)\,x\,,
\end{equation}
and we extend it to $\F$-adapted step processes by linearity.
\begin{theorem}[\Ito$\;$ isomorphism]
\label{th:Ito}
Let $T>0$, $p\in (0,\infty)$ and let $X$ be a UMD Banach space, then the mapping $G\rightarrow \int_0^{T} G\,dW_H$ admits a unique extension to a isomorphism from $L^p_{\F}(\Omega;\g(0,T;H,X))$ into $L^p(\Omega;X)$ and
\begin{equation*}
\E\sup_{0\leq  t\leq T}\Big\|\int_0^t G(s)\,dW_H(s)\Big\|_{X}^p \eqsim_{p,X} \E\|G\|_{\g(0,T;H,X)}^p.
\end{equation*}
\end{theorem}
If $G$ does not depend on $\Omega$, then the above holds for every Banach space $X$ and the norm equivalence only depends on $p\in (0,\infty)$.

For future references, we make the following simple observation. To state this, we denote by $L^p_{\F}(\Omega\times (0,T);\g(H,X))$ the closure in $L^p(\Omega\times (0,T);\g(H,X))$ of all simple $\F$-adapted stochastic process.

As a consequence of Proposition \ref{prop:inclusiongammatype} one easily obtains the following:
\begin{corollary}
\label{cor:Ito}
Let $T>0$, $p\in (0,\infty)$ and let $X$ be a UMD Banach space with type 2. Then the mapping $G\mapsto \int_0^T G\,dW_H$ extends to a bounded linear operator from $L^p_{\F}(\Omega\times (0,T);\g(H,X))$ into $L^p(\Omega;X)$. Moreover,
\begin{equation*}
\E\sup_{0\leq  t\leq T}\Big\|\int_0^t G(s)\,dW_H(s)\Big\|_{X}^p \lesssim_{p,X,T} \E\|G\|_{L^2(0,T;\g(H,X))}^p.
\end{equation*}
\end{corollary}

\section{Stochastic Maximal $L^p$-regularity}
\label{s:SMR}
Throughout the rest of the paper we assume that the operator $-A$ with domain $D(A)$ is a closed operator and generates a strongly continuous semigroup $(S(t))_{t\geq 0}$ on a Banach space $X$ with UMD and type $2$.

\subsection{Solution concepts}
\label{ss:solutionconcepts}
For processes $F\in L^1_{\F}(\O\times (0,T);X)$ and $G\in L^2_{\F}(\O\times (0,T);\g(H,X))$ for every $T<\infty$, consider the following stochastic evolution equation
\begin{equation}\label{eq:SEEGdW}
\begin{cases}
dU + A U\ud t = F\ud t+ G \ud W_H, & \text{on $\R_+$},\\
U(0)=0.
\end{cases}
\end{equation}
The {\em mild solution} to \eqref{eq:SEEGdW} is given by
\[U(t) = S*F(t) + S\diamond G(t):=\int_0^t S(t-s) F(s)\ud s+\int_0^t S(t-s) G(s)\,dW_H(s).\]
for $t\geq 0$.
It is well-known that the mild solution is a so-called {\em weak solution} to \eqref{eq:SEEGdW}:
for all $x^*\in D(A^*)$, for all $t\geq0$, a.s.
\[\lb U(t),x\rb + \int_0^t \lb U(s),A^*x^*\rb \ud s = \int_0^t \lb F(s),x^*\rb \ud s+\int_0^t G(s)^*x^* \ud W_H(s)\]
Conversely, if $U\in L^1_{\rm loc}(\R_+;X)$ a.s. is a weak solution to \eqref{eq:SEEGdW}, then $U$ is a mild solution. Moreover, if $U\in L^1_{\rm loc}(\R_+;D(A))$, then additionally $U$ is a {\em strong solution} to \eqref{eq:SEEGdW}: for all $t\geq0$ a.s.
\[U(t) + \int_0^t A U(s) \ud s = \int_0^t F(s)\ud s + \int_0^t G(s) \ud W_H(s).\]
For details we refer to \cite{DPZ} and \cite{VThesis}.

\subsection{Main definitions}
\begin{definition}[Stochastic maximal $L^p$-regularity]
\label{def:RegJ}
Let $X$ be a UMD space with type $2$, let $p\in [2, \infty)$, $w>\omega_0(-A)$ and let $J = (0,T)$ with $T\in (0,\infty]$. The operator $A$ is said to have {\em stochastic maximal $L^p$-regularity} on $J$ if for each $G\in L^p_{\F}(\O\times J;\g(H,X))$ the stochastic convolution
$S\diamond G$ takes values in $D((w+A)^{1/2})$ $\P\times dt$-a.e., and satisfies
\begin{equation}\label{eq:SGdef}
\|S\diamond G\|_{L^p(\O\times J;D((w+A)^{1/2}))}\leq C \|G\|_{L^p(\O\times J;\g(H,X))},
\end{equation}
for some $C>0$ independent of $G$. In this case we write $A\in \RegJ$.
\end{definition}
Note that, the class $\RegJ$ does not depend on $w>\omega_0(-A)$. Indeed, for any $w,w'> \omega_0(-A)$, $D((w+A)^{1/2}) = D((w'+A)^{1/2})$ isomorphically.

Some helpful remarks may be in order.

\begin{remark}
\label{r:Stochasticintegralwelldef}
In Definition \ref{def:RegJ} it suffices to consider $G$ in a dense class of a subset of $L^p_{\F}(\O\times\R_+;\g(H,X))$ for which the stochastic convolution process $(w+A)^{1/2}S\diamond G(t)$ is well-defined for each $t\geq 0$. For example, the set of all adapted step processes with values in $D(A)$ (or the space $L^p_{\F}(\Omega\times J;\g(H,D(A)))$) can be used.
Indeed, if $G \in L^p_{\F}(\Omega\times J;\g(H,D(A)))$, then $s\mapsto (w+A)^{1/2} S(t-s) G(t)$ belongs to $ L^p(\Omega\times J;\g(H,X))$ for each $t\in J$. Indeed, for $t\in J$,
\begin{align*}
\E\int_0^t \|(w+A)^{1/2} S(t-s) G(s)\|^p_{\g(H,X)}\, ds &\leq M^2 \E\int_0^t \|(w+A)^{1/2} G(s)\|^p_{\g(H,X)}\, ds
\\ & \leq c\, M^2 \|G\|_{L^p(\Omega\times J;\g(H,D(A)))},
\end{align*}
where $M:=\sup_{s\leq t}\|S(t)\|$. Therefore, for each $t\in J$, the well-definedness of $(w+A)^{1/2} S\diamond G(t)$ follows from Corollary \ref{cor:Ito}.
\end{remark}

\begin{remark}
\label{r:RegJdelta}
In the setting of Definition \ref{def:RegJ}, for $\alpha\in[1/2,1]$, one could ask for
\begin{equation}\label{eq:SGdefdelta}
\|S\diamond G\|_{L^p(\O\times J;D((w+A)^{\alpha}))}\leq C \|G\|_{L^p(\O\times J;\g(H,D((w+A)^{\alpha-\frac12})))},
\end{equation}
for each $G\in L^p_{\F}(\O\times J;\g(H,D((w+A)^{\alpha-\frac12})))$. One can easily deduce that $A$ satisfies \ref{eq:SGdefdelta} if and only if $A\in\RegJ$.
\end{remark}

Before going further, we introduce an homogeneous version of stochastic maximal $L^p$-regularity:
\begin{definition}[Homogeneous Stochastic Maximal $L^p$-regularity]\label{def:RegJw}
Let $X$ be a UMD space with type $2$ and let $p\in [2, \infty)$. The operator $A$ is said to have  {\em homogeneous stochastic maximal $L^p$-regularity} if for each $G\in L^p_{\F}(\O\times \R_+;\g(H,X))$ the stochastic convolution $S\diamond G$ takes values in $D(A^{1/2})$ $\P\times dt$-a.e.\ and
\begin{equation}\label{eq:SGdefw}
\|A^{1/2} S\diamond G\|_{L^p(\O\times \R_+;X)}\leq C \|G\|_{L^p(\O\times \R_+;\g(H,X))},
\end{equation}
for some $C>0$ independent of $G$. In this case we write $A\in \RegRw$.
\end{definition}
There is no need for the homogeneous version of $\RegJ$ for $J=(0,T)$ with $T<\infty$, since in this situation by Corollary \ref{cor:Ito} we have
$$\|S\diamond G\|_{L^p(\Omega\times J;X)} \leq c_T \|G\|_{L^p(\Omega\times J;\g(H,X))}.$$
Moreover, it is clear that if $A\in \RegRw$ for some $p\in[2,\infty)$ and $0\in \rho(A)$ (thus $0\in \rho(A^{1/2})$) then $A\in\RegR$.
The converse is also true as Corollary \ref{cor:hom} below shows.

We will mainly study the class $\RegJ$ (for $T\in (0,\infty]$). However, many results can be extended to the class $\RegRw$ without difficulty.

In order to state the following result we introduce the following condition:
\begin{assumption}\label{ass:JRbdd}
Let $X$ be a UMD Banach space with type $2$ and let $p\in [2, \infty)$. Assume that the following family is $R$-bounded
\begin{equation*}
\{J_{\delta}\}_{\delta>0} \subseteq \calL (L^p_{\F}(\Omega\times\R_+;\g(H,X)),L^p(\Omega\times\R_+;X)),
\end{equation*}
where
$J_{\delta}f(t):=\frac{1}{\sqrt{\delta}}\int_{(t-\delta)\vee 0}^t f(s)dW_H(s)$.
\end{assumption}
The above holds for $p\in (2, \infty)$ if $X$ is isomorphic to a closed subspace of an $L^q(S)$ space with $q\in [2, \infty)$. If $q=2$, one can also allow $p=2$.
The following central result was proved in \cite{MaximalLpregularity,NVW11,NVW13}; see also Remark \ref{r:SMRthetafromHinfinite}.
\begin{theorem}\label{thm:SMRmain}
Suppose that Assumption \ref{ass:JRbdd} is satisfied. If $A$ has a bounded $H^\infty$-calculus with $\omega_{H^\infty}(A)<\pi/2$, then $A\in \RegRw$.
\end{theorem}

\subsection{Deterministic characterization and immediate consequences}
In the next proposition we make a first reduction to the case where $G$ does not depend on $\O$.
\begin{proposition}\label{prop:omegaind}
Let $X$ be a UMD space with type $2$, let $p\in [2, \infty)$, let $J=(0,T)$ with $T\in (0,\infty]$ and fix $w>\omega_0(-A)$. Then the following are equivalent:

\let\ALTERWERTA\theenumi
\let\ALTERWERTB\labelenumi
\def\theenumi{(H1)}
\def\labelenumi{(H1)}
\begin{enumerate}[{\rm (1)}]
\item\label{it:omegaind1} $A\in \RegJ$.
\item\label{it:omegaind2} There exists a constant $C$ such that for all $G\in L^p(J; \g(H,D(A)))$,
\begin{equation*}
\Big(\int_0^T \|s\mapsto (w+A)^{1/2}S(t-s)G(s)\|_{\g(0,t;H,X)}^p \, dt\Big)^{1/p} \leq C \|G\|_{L^p(J;\g(H,X))}.
\end{equation*}
\end{enumerate}
\end{proposition}
\begin{proof}

\eqref{it:omegaind1} $\Rightarrow$ \eqref{it:omegaind2}: For $G\in L^p(J; \g(H,D(A)))$, Theorem \ref{th:Ito} provides the two-sides estimates
\[\|(w+A)^{1/2} S\diamond G(t)\|_{L^p(\O;X)} \eqsim_{p,X} \|s\mapsto (w+A)^{1/2}S(t-s)G(s)\|_{\g(0,t;H,X)}.\]
Now the claim follows by taking $L^p(J)$-norms in the previous inequalities.

\eqref{it:omegaind2} $\Rightarrow$ \eqref{it:omegaind1}:  As in the previous step, we employ Theorem \ref{th:Ito}. Indeed, for any $t\in J$ and $G$ an adapted step process, we have
\begin{equation*}
\|(w+A)^{1/2}S\diamond G(t)\|_{L^p(\Omega;X)}^p \eqsim_{p,X} \E\|s\mapsto (w+A)^{1/2}S(t-s)G(s)\|^p_{\g(0,t;H,X)}.
\end{equation*}
Integrating over $t\in J$, we get
\begin{equation*}
\begin{aligned}
\|(w+A)^{1/2} S\diamond G\|_{L^p(\Omega\times J;X)}^p &\eqsim_{X,p}  \E \int_0^T \|s\mapsto (w+A)^{1/2}S(t-s)G(s)\|_{\g(0,t;H,X)}^p\,dt\\
&\leq C^p \E\int_0^T\|G(t)\|_{\g(H,X)}^p \,dt=C^p \|G\|_{L^p(\Omega\times J;\g(H,X))}^p,
\end{aligned}
\end{equation*}
where in the last we have used the inequality in \eqref{it:omegaind2} pointwise in $ \Omega$. The claim follows by density of the adapted step process in $L^p_{\F}(\Omega\times J; \g(H,X))$.
\end{proof}

\begin{proposition}\label{prop:trans}
Let $X$ be a UMD space with type $2$, let $p\in [2, \infty)$. Let $J = (0,T)$ with $T\in (0,\infty]$ and assume $A\in \RegJ$. Then:
\begin{enumerate}[{\rm (1)}]
\item\label{it:trans1} If $T<\infty$ and $\lambda\in \C$, then $A+\lambda\in \RegJ$.
\item\label{it:trans2} If $T=\infty$ and $\lambda\in \C$ is such that $\Re\lambda\geq 0$, then $A+\lambda\in \RegR$.
\item\label{it:trans3} If $T\in (0,\infty]$ and $\lambda>0$, then $\lambda A\in \RegJlambda$.
\end{enumerate}
\end{proposition}
\begin{proof}
\eqref{it:trans1}: Note that $-A-\lambda$ generates $(e^{-\lambda t}S(t))_{t> 0}$. Then, fix $w>\omega_0(-A-\lambda)$ (thus $w+\lambda>\omega_0(-A)$) and let $G\in L^p(J; \g(H,D(A)))$.
By \eqref{eq:phiGinfty} one has
\begin{align*}
\|s\mapsto (w+\lambda+A)^{1/2} & e^{-\lambda (t-s)}S(t-s)G(s)\|_{\g(0,t;H,X)} \\ & \leq
M_{T,\lambda} \|s\mapsto (w+\lambda+A)^{1/2}S(t-s)G(s)\|_{\g(0,t;H,X)},
\end{align*}
where $M_{T,\lambda}=\sup_{\{0<s<t<T	\}}e^{-(\text{Re}\lambda)(t-s)}$.
Therefore, taking the $L^p(J)$-norms, Proposition \ref{prop:omegaind} implies the required result.

\eqref{it:trans2}: Follows by the same argument of (1) but in this case $M_{\infty,\lambda}=\sup_{\{0<s<t\}}e^{-(\text{Re}\lambda)(t-s)}$ is finite if and only if $\Re\lambda>0$.

\eqref{it:trans3}: Note that $-\lambda A$ generates $(S(\lambda t))_{t>0}$. Fix $G\in L^p(0,T/\lambda; \g(H,D(A)))$ and $w>\omega_0(-\lambda A)$ (thus $w/\lambda>\omega_0(-A)$), one has
\begin{align*}
\|s\mapsto (w+\lambda A)^{1/2} & S(\lambda (t-s))G(s)\|_{\g(0,t;H,X)} \\
& =\|s\mapsto (w+\lambda A)^{1/2} S(\lambda s)G(t-s)\|_{\g(0,t;H,X)}\\
&\eqsim_{\lambda}\|s\mapsto (\frac{w}{\lambda} + A)^{1/2}S(s)G(t-\frac{s}{\lambda})\|_{\g(0,\lambda t;H,X)}.
\end{align*}
Then integrating over $0<t< T/\lambda$, one has
\begin{align*}
\int_0^{\frac{T}{\lambda}}\|s\mapsto (w+&\lambda A)^{1/2}  S(\lambda (t-s))G(s)\|_{\g(0,t;H,X)}^p \,dt \\
&\eqsim_{\lambda}
\int_0^{\frac{T}{\lambda}}\|s\mapsto (\frac{w}{\lambda} + A)^{1/2}S(s)G(t-\frac{s}{\lambda})\|_{\g(0,\lambda t;H,X)}^p\,dt\\
&\eqsim_{\lambda}
\int_0^{T}\|s\mapsto (\frac{w}{\lambda} + A)^{1/2}S(s)G(\frac{\tau-s}{\lambda})\|_{\g(0,\tau;H,X)}^p\,d\tau\\
&\leq C_{\lambda,p,A} \int_0^{T} \|G(\frac{s}{\lambda})\|_{\g(H,X)}^p ds =C_{\lambda,p,A} \int_0^{\frac{T}{\lambda}} \|G({s})\|_{\g(H,X)}^p ds;
\end{align*}
where in the last inequality we have used that $A\in \RegJ$. Thus Proposition \ref{prop:omegaind} ensures that $\lambda A\in \RegJlambda$.
\end{proof}
In Corollary \ref{cor:T1T2} we will see a refinement of Proposition \ref{prop:trans}.

\subsection{Independence of $H$}
\label{s:independenceH}
\begin{theorem}
\label{th:independenceH}
Let $X$ be a UMD space with type $2$, let $p\in [2, \infty)$ and let $J=(0,T)$ with $T\in (0,\infty]$. The following are equivalent:
\begin{enumerate}[{\rm (1)}]
\item $A\in \RegJ$ for $H = \R$.
\item $A\in \RegJ$ for any Hilbert space $H$.
\end{enumerate}
\end{theorem}
\begin{proof}
It suffices to prove (1)$\Rightarrow$(2), since the converse is trivial. Assume (1) holds. Without loss of generality we can assume $H$ is separable (see \cite[Proposition 9.1.7]{Analysis2}).
Let $\Gamma:\R_+\to L^p(\wt{\Omega};X)$ be defined by $\Gamma(s) = \sum_{n\geq 1} \gamma_n G(s) h_n$, where $(h_n)_{n\geq 1}$  is an orthonormal basis for $H$.  Then by the Kahane--Khincthine inequalities and the definition of the $\gamma$-norm we have
\begin{align}\label{eq:gammaKK}
\|G(s)\|_{\gamma(H,X)} =\|\Gamma(s)\|_{L^2(\tilde{\O};X)} \eqsim_p  \|\Gamma(s)\|_{L^p(\tilde{\O};X)}.
\end{align}
By Proposition \ref{prop:inclusiongammatype}
\begin{align*}
\|s\mapsto (w+A)^{1/2}& S(t-s)G(s)\|_{\g(0,t;H,X)}
\\ & \lesssim_{X} \|s\mapsto (w+A)^{1/2}S(t-s)G(s)\|_{\g(0,t;\gamma(H,X))}
\\ & = \|s\mapsto (w+A)^{1/2}S(t-s)\Gamma(s)\|_{\g(0,t;L^2(\tilde{\Omega};X))}
\\ & \stackrel{(*)}{=} \|s\mapsto (w+A)^{1/2}S(t-s)\Gamma(s)\|_{L^2(\tilde{\Omega};\g(0,t;X))}
\\ & \leq \|s\mapsto (w+A)^{1/2}S(t-s)\Gamma(s)\|_{L^p(\tilde{\Omega};\g(0,t;X))},
\end{align*}
where we applied the $\gamma$-Fubini's theorem (see \cite[Theorem 9.4.8]{Analysis2}) in $(*)$.
By Fubini's theorem and Proposition \ref{prop:omegaind} we obtain
\begin{align*}
\int_J \|s\mapsto &(w+A)^{1/2} S(t-s)G(s)\|_{\g(0,t;H,X)}^p \ud t \\ &\leq \tilde{\E}\int_J \|s\mapsto (w+A)^{1/2}S(t-s)\Gamma(s)\|_{\g(0,t;X)}^p \ud t
\\ & \leq C^p \tilde{\E}\|\Gamma\|_{L^p(J;X)}^p = C^p \|\Gamma\|_{L^p(J;L^p(\tilde{\O};X))}^p \eqsim_p C^p \|G\|_{L^p(J;\g(H,X))}.
\end{align*}
where in "$\eqsim_p$" we used \eqref{eq:gammaKK}. Now the result follows from Proposition \ref{prop:omegaind}.
\end{proof}

\section{Analyticity and exponential stability\label{sec:analyticity}}
\label{s:analytic}

The main result of this section is the following.
\begin{theorem}\label{thm:analytic}
Let $X$ be a Banach space with UMD and type $2$ and let $p\in [2, \infty)$. Let $J = (0,T)$ with $T\in (0,\infty]$.
If $A\in \RegJ$, then $-A$ generates an analytic semigroup.
\end{theorem}
The proof consists of several steps and will be explained in the next subsections.

\subsection{Square function estimates}
Next we derive a simple square function estimates from $\RegJ$. In order to include the case $T=\infty$ we need a careful analysis of the constants.
\begin{lemma}\label{lem:gammaest}
Let $X$ be a UMD space with type $2$, let $p\in [2, \infty)$, let $J= (0,T)$ with $T\in (0,\infty]$ and let $w>\omega_0(-A)$. If $A\in \RegJ$, then there is a constant $C$ such that for all $x\in X$,
\begin{equation}\label{eq:gammaest}
\|s\mapsto (w+A)^{1/2}S(s) x\|_{\g(J;X)}\leq C\|x\|.
\end{equation}
\end{lemma}
\begin{proof}
First assume $T<\infty$ and fix $h\in H$ with $\|h\|=1$. Let $G\in L^p(J;\g(H,X))$ be given by $G(t) = \one_{J} h\otimes x$. Then  for $t\in [T/2,T]$ one can write
\begin{align*}
\|s\mapsto (w+A)^{1/2}S(s)x \|_{\g(0,T/2;X)} & \leq \|s\mapsto (w+A)^{1/2}S(s)x \|_{\g(0,t;X)} \\ & = \|s\mapsto (w+A)^{1/2}S(t-s)x \|_{\g(0,t;X)} \\ & = \|s\mapsto (w+A)^{1/2}S(t-s)G(s)\|_{\g(0,t;H,X)}.
\end{align*}
Therefore, taking $p$-th powers on both sides integration over $t\in J$, and applying Proposition \ref{prop:omegaind}  yields
\begin{align*}
T \|s\mapsto (w+A)^{1/2}S(s)x \|_{\g(0,T/2;X)}^p & \leq  \int_0^{T} \|s\mapsto (w+A)^{1/2}S(t-s)G(s)\|_{\g(0,t;H,X)}^p \, dt
\\ & \leq C^p \|G\|_{L^p(J;\g(H,X))}^p = C^p T \|x\|^p.
\end{align*}
Therefore,
\begin{equation}\label{eq:ineqSx}
\|s\mapsto (w+A)^{1/2}S(s)x \|_{\g(0,T/2;X)} \leq C \|x\|, \ \ x\in X.
\end{equation}
By the left-ideal property and \eqref{eq:ineqSx} we see that
\begin{align*}
\|s\mapsto (w+A)^{1/2}S(s)x \|_{\g(T/2,T;X)} & = \|s\mapsto S(\tfrac{T}{2}) (w+A)^{1/2}S(s-\tfrac{T}{2})x \|_{\g(T/2,T;X)}
\\ & \leq \|S(\tfrac{T}{2})\| \, \|s\mapsto (w+A)^{1/2}S(s)x \|_{\g(0,T/2;X)} \\ & \leq C  \|S(\tfrac{T}{2})\|  \, \|x\|.
\end{align*}
Combining this with \eqref{eq:ineqSx} and \eqref{eq:Geq} yields
\begin{align*}
\|(w+A)^{1/2} &S(s) x\|_{\g(J;X)} \\
 &\leq\|(w+A)^{1/2} S(s) x\|_{\g(0,T/2;X)} + \|(w+A)^{1/2} S(s) x\|_{\g(T/2,T;X)}\\
 &\leq C_{S,T} \|x\|.
\end{align*}

Next we consider $T=\infty$. Applying Proposition \ref{prop:omegaind} with $G \one_{[0,R]}$ with $R>0$ fixed and \eqref{eq:phiGinfty} gives that
\begin{align*}
\Big(\int_0^{R} \|s\mapsto (w+A)^{1/2}S(t-s)G(s)\|_{\g(0,t;H,X)}^p \, dt\Big)^{1/p} \leq C \|G\|_{L^p(0,R;\g(H,X))},
\end{align*}
where $C$ is independent of $R$. Therefore, arguing as in \eqref{eq:ineqSx} we obtain that for all $R<\infty$,
\[\|s\mapsto (w+A)^{1/2}S(s)x \|_{\g(0,R/2;X)} \leq C \|x\|.\]
The result now follows since (see \cite[Proposition 2.4]{NVW1})
\[\|s\mapsto (w+A)^{1/2}S(s)x \|_{\g(\R_+;X)} = \sup_{R>0}\|s\mapsto (w+A)^{1/2}S(s)x \|_{\g(0,R/2;X)}.\]
\end{proof}
Choosing $w=0$ in (\ref{lem:gammaest}) in Lemma \ref{lem:gammaest}, we obtain the following:
\begin{corollary}
Suppose that $A\in \RegR$, $\omega_0(-A)<0$ and set $\varphi(z):=z^{1/2}e^{-z}$, then there exists a constant $c>0$ such that
\begin{equation*}
\|t\mapsto \varphi(tA)x\|_{\g(\R_+,\frac{dt}{t};X)}\leq c \|x\|\,,
\end{equation*}
for all $x\in X$.
\end{corollary}

\subsection{Sufficient conditions for analyticity}

To prove Theorem \ref{thm:analytic}  we need several additional results which are of independent interest. The next result is a comparison result between $\gamma$-norms and $L^p$-norms of certain orbits for spaces with cotype $p$. Related estimates for general analytic functions can be found in \cite[Theorem 4.2]{VerWeLPS}, but are not applicable here.

\begin{lemma}\label{lem:Lpgammaorbit}
Let $X$ be a Banach space with cotype $p$. Let $\omega_0(-A)<0$. Then for all $q>p$ there exists a $C>0$ such that for all $x\in D(A^2)$,
\[
\|t\mapsto A^{1/q}S(t) x\|_{L^q(\R_+;X)} \leq C \|t\mapsto A^{1/2}S(t) x\|_{\g(\R_+;X)}.
\]
Moreover, if $p=2$, then one can take $q=2$ in the above.
\end{lemma}
The right-hand side of the above estimate is finite. Indeed, for $x\in D(A^2)$, we have $A^{1/2}S(\cdot) x = S(\cdot) A^{1/2}x\in C^1([0,T];X)$, thus it follows from \cite[Proposition 9.7.1]{Analysis2} that $A^{1/2}S(\cdot) x\in \gamma(0,T;X)$. Now since $S$ is exponentially stable we can conclude from \cite[Proposition 4.5]{NWa} that $A^{1/2}S(\cdot) x\in \gamma(\R_+;X)$.

\begin{proof}
By an approximation argument we can assume $x\in D(A^3)$.
Let $(\phi_n)_{n\geq 0}$ be a Littlewood-Paley partition of unity as in \cite[Section 6.1]{BeLo}. Let $f:\R\to X$ be given by $f(t):=A^{1/q}S(|t|) x$. Then $f'(t) = \sign(t) A f(t)$ for $t\in \R\setminus\{0\}$. Let $f_n:=\phi_n*f$ for $n\geq 0$. Let $\psi$ be such that $\wh{\psi} = 1$ on $\supp \wh{\phi}_1$ and $\wh{\psi}\in C^\infty_c(\R\setminus\{0\})$. Set $\wh{\psi}_n(\xi) = \wh{\psi}_1(2^{-(n-1)}\xi)$ for $n\geq 1$. Then $f_n = \psi_n*f_n$.

{\em Step 1:}
We will first show that for all $\alpha\in (0,1)$, there is a constant $C$ such that for all $n\geq 0$
\begin{align}\label{eq:estfnorbitLp}
\|f_n\|_{p}&\leq C 2^{-\alpha n} \|A^{\alpha} f_n\|_{p},
\end{align}
where we write $\|\cdot\|_p := \|\cdot\|_{L^p(\R;X)}$.
As a consequence the estimate \eqref{eq:estfnorbitLp} holds for an arbitrary $\alpha>0$ if one takes $x\in D(A^{r+2})$ (where $\alpha<r\in \N$). For $n=0$ the estimate is clear from $0\in \rho(A^{\alpha})$. To prove the estimate for $n\geq 1$ note that by the moment inequality (see \cite[Theorem II.5.34]{EN}) and H\"older inequality,
\begin{equation}\label{eq:momentcons}
\|A f_n\|_{p}\leq C \|A^{\alpha} f_n\|_p^{\frac{1}{2-\alpha}} \|A^2 f_n\|_p^{\frac{1-\alpha}{2-\alpha}}.
\end{equation}
Using $f_n = \psi_n*f_n$ and the properties of $S$ we obtain
\begin{equation}
\label{eq:f'sign}
\sign(\cdot) A f_n  = \frac{d}{dt}f_n = \psi_n'*f_n.
\end{equation}
Therefore, by Young's inequality
\begin{align*}
\|A^2 f_n\|_p = \|\psi_n'*A f_n\|_p \leq \|\psi_n'\|_1 \|A f_n\|_p \leq C_{\psi} 2^{n} \|A f_n\|_p.
\end{align*}
Combining this with \eqref{eq:momentcons} we obtain
\begin{align}\label{eq:Afnestalpha}
\|A f_n\|_{p}\leq C 2^{n(1-\alpha)} \|A^{\alpha} f_n\|_p.
\end{align}

Next we prove an estimate for $\|f_n\|_p$. Let $d_t = \frac{d^2}{dt^2}$ and set $J_{\beta} = (1-d^2_t)^{\beta/2}$ for $\beta\in \R$. Then $J_{\beta_1}J_{\beta_2} = J_{\beta_1+\beta_2}$ for $\beta_1, \beta_2\in \R$. Recall from the proof of \cite[Theorem 6.1]{Am97} that for any $g\in L^p(\R;X)$ and $\beta\in \R$, we have
\[\|J_{\beta} \psi_n*g\|_p\leq C_{\beta,\psi} 2^{\beta n}  \|\psi_n*g\|_p.\]
Therefore,
\[\|f_n\|_p  =\|\psi_n*\varphi_n*f\|_p =  \|J_{-2} \psi_n*(J_2 \varphi_n) * f\|_p \leq C_{\psi}  2^{-2n} \|\psi_n*(J_2 \varphi_n) * f\|_p. \]
Now since $J_2 = 1-d_t^2$ we can estimate
\begin{equation*}
\begin{aligned}
\|\psi_n*(J_2 \varphi_n) * f\|_p  &\leq \|\psi_n*\varphi_n * f\|_p + \|d_t^2(\psi_n *\varphi_n * f)\|_p \\
&\leq C_{\psi} \|f_n\|_p + \|\psi_n'* \varphi_n* f'\|_p.
\end{aligned}
\end{equation*}
By Young's inequality
\[\|\psi_n'* \varphi_n*f' \|_p\leq \|\psi_n'\|_1 \|\varphi_n*f'\|_p \leq C_{\psi}2^n \|(f_n)'\|_p=C_{\psi}2^n \|Af_n\|_p,\]
where in the last equality we have used \eqref{eq:f'sign}. Thus we can conclude
\begin{align}\label{eq:fnbyAfn}
\|f_n\|_p \leq C_{\psi} 2^{-n}  (\|f_n\|_p +  \|A f_n\|_p)\leq C_{\psi,A}  2^{-n}  \|A f_n\|_p,
\end{align}
where in the last step we used the fact that $A$ is invertible.

Now \eqref{eq:estfnorbitLp} follows by combining \eqref{eq:Afnestalpha} and \eqref{eq:fnbyAfn}.

\medskip

{\em Step 2:} By Step 1 with $\alpha := \frac{1}{2} - \frac1q$ and \cite[Lemma 4.1]{RozVer17} we can estimate
\begin{align*}
\|f_n\|_{p} \leq C 2^{-n \alpha}\|A^{\alpha} f_n\|_p\leq C_{p,X} 2^{-n \alpha} 2^{\frac{n}{2} - \frac{n}p}\|A^{\alpha} f_n\|_{\gamma(\R;X)}.
\end{align*}
Multiplying by $2^{\frac{n}{p} - \frac{n}{q}}$ and taking $\ell^p$-norms and applying \cite[Lemma 2.2]{KNVW} in the same way as in \cite[Theorem 1.1]{KNVW} gives
\begin{align*}
\|f\|_{B^{\frac1p-\frac1q}_{p,p}(\R;X)} & \leq C_{p,X} (\sum_{n\geq0} \|A^{\alpha} f_n\|_{\gamma(\R;X)}^p)^{1/p} \\ & \leq C_{p,X}'   \|A^{\alpha} f\|_{\gamma(\R;X)} \leq 2C_{p,X}' \|t\mapsto A^{1/2} S(t) x\|_{\g(\R_+;X)},
\end{align*}
where in the last step we used \eqref{eq:Geq}.

It remains to note that $B^{\frac1p-\frac1q}_{p,p}(\R;X)\hookrightarrow L^q(\R;X)$ (see \cite[Theorem 1.2 and Proposition 3.12]{MeyVer12}).

The final assertion for $p=2$ is immediate from Proposition \ref{prop:inclusiongammatype}.
\end{proof}

Next we show that certain $L^p$-estimates for orbits implies analyticity of the semigroup $S$. 
\begin{lemma}\label{lem:analyticLp}
Let $X$ be a Banach space and let $w>\omega_0(-A)$. If for some $T\in (0,\infty]$, $C>0$, $p\geq 2$, the operator $A$ satisfies
\begin{equation}\label{eq:ASx}
\|t\mapsto (w+A)^{1/p}S(t) x\|_{L^p(0,T;X)} \leq C \|x\|_{X},  \ \ x\in D(A),
\end{equation}
then $-A$ generates an analytic semigroup.
\end{lemma}
It seems that the above result was first observed in \cite[Proposition 2.7]{BounitDrEl}. The proof below is different and was found independently.

\begin{proof}
Clearly, we can assume $T<\infty$. Moreover, without loss of generality, one can reduce to the case that $S$ is exponentially stable and $w=0$. Finally, we can also assume that $p\geq 2$ is an integer. Indeed, fix $n\in \N$ such that $n\geq p$. By the moment inequality (see \cite[Theorem II.5.34]{EN}) for all $t\in [0,T]$, we have
\begin{align*}
\|(w+A)^{1/n}S(t) x\|^n &\lesssim_{n,p,A,w} \|S(t) x\|^{n-p} \|(w+A)^{1/p}S(t) x\|^{p}
\\ & \lesssim_{n,p,A,T} \|x\|^{n-p} \|(w+A)^{1/p}S(t) x\|^{p},
\end{align*}
where $C, \tilde{C}$ only dependent on $n, p, A, T, w$. Therefore,
\begin{align*}
\int_0^T \|(w+A)^{1/n}S(t) x\|^n \ud t\lesssim_{n,p,A,T,w} \|x\|^{n-p} \int_0^T \|(w+A)^{1/p}S(t) x\|^p \ud t \leq C^n \|x\|^n.
\end{align*}

To prove that $(S(t))_{t\geq 0}$ is analytic, it suffices by \cite[Theorem II.4.6]{EN} to show that   $\{t A S(t): t\in (0,T]\}\subseteq \calL(X)$ is bounded. To prove this fix $x\in D(A)$. Let $M = \sup_{t\geq 0} \|S(t)\|$. Let $t_n = \frac{T}{p 2^n}$ for $n\geq 0$. Then for all $t\in [t_{n+1}, t_n]$ we have
$\|A^{1/p} S(t_n)x\| \leq M \|A^{1/p} S(t)x\|$ and thus integration gives
\begin{align*}
\frac12 t_n \|A^{1/p} S(t_n)x\|^p & = (t_{n}-t_{n+1}) \|A^{1/p} S(t_n)x\|^p \\ & \leq M^p \int_J \|A^{1/p} S(t)x\|^p \ud t \leq M^p C^p \|x\|^p.
\end{align*}
Now fix $t\in (0,T/p]$. Choose $n\geq 0$ such that $t\in [t_{n+1}, t_n]$. Then we obtain
\begin{align*}
t \|A^{1/p} S(t)x\|^p \leq 2 M^p t_{n+1} \|A^{1/p} S(t_{n+1})x\|^p \leq 4 M^{2p} C^p \|x\|^p.
\end{align*}
By density it follows that $S(t):X\to D(A^{1/p})$ is bounded and
$t^{1/p} \|A^{1/p} S(t)\|\leq 4^{1/p} M^2 C$ for each $t\in (0,T/p]$.
We can conclude that for all $t\in (0,T]$,
\[\|t A S(t)\| = \|(t^{1/p} A^{1/p} S(t/p))^p\| \leq t \|A^{1/p} S(t/p)\|^p\leq  4p M^{2p} C^p.\]
\end{proof}

\begin{proposition}\label{prop:analyticgamma}
Let $X$ be a Banach space with finite cotype. Let $J = (0,T)$ with $T\in (0,\infty]$. Let $w>\omega_0(-A)$. If there exists a $c>0$ such that
\begin{equation}\label{eq:ASxgamma}
\|t\mapsto (w+A)^{1/2}S(t) x\|_{\g(J;X)} \leq c \|x\|,  \ \ x\in X,
\end{equation}
then $-A$ generates an analytic semigroup.
\end{proposition}
\begin{proof}
By rescaling we can assume that $S$ is exponentially stable, thus we may take $w=0$. Moreover, by \cite[Proposition 4.5]{NWa} we can assume $T=\infty$.  Now the result follows by combining Lemmas  \ref{lem:Lpgammaorbit} and \ref{lem:analyticLp}.
\end{proof}

\begin{proof}[Proof of Theorem \ref{thm:analytic}]
By Lemma \ref{lem:gammaest} the estimate \eqref{eq:ASxgamma} holds. Moreover, since $X$ has type $2$, it has finite cotype (see \cite[Theorem 7.1.14]{Analysis2}). Therefore, by Proposition \ref{prop:analyticgamma}, $-A$ generates an analytic semigroup.
\end{proof}

From the proof of Theorem \ref{thm:analytic} we obtain the following.
\begin{remark}
Assume $A\in \RegJ$, $\omega_0(-A)<0$ and $X$ has cotype $p_0$. Let $p>p_0$. Then there is a constant $C$ such that for all $x\in X$,
\[\int_{\R_+} \|A^{1/p} S(t) x\|^p \ud t\leq C^p \|x\|^p.\]
This type of estimate gives the boundedness of some singular integrals.
\end{remark}

\subsection{Exponential stability}

\begin{proposition}[Stability]
\label{prop:stability}
Let $X$ be a UMD space with type $2$, let $p\in [2, \infty)$. If $A\in \RegR$, then $\omega_0(-A)<0$.
\end{proposition}
\begin{proof}
Let $w>\omega_0(-A)$. Let $y\in X$ be arbitrary. Taking $x=(w+A)^{-1/2} y$ in Lemma \ref{lem:gammaest} one obtains
\[\|s\mapsto S(s) y\|_{\g(\R_+;X)}\leq C\|(w+A)^{-1/2} y\|\leq C' \|y\|.\]
Thus from \cite[Theorem 3.2]{HaNeVe} it follows that there is an $\varepsilon>0$ such that $\{(\lambda+A)^{-1}: \lambda>-\varepsilon\}$ is uniformly bounded.
From Theorem \ref{thm:analytic} it follows that $A$ generates an analytic semigroup, and hence  $0>s_0(-A) = \omega_0(-A)$ (see \cite[Corollary IV.3.12]{EN}).
\end{proof}

As announced in Section \ref{s:SMR} we now can prove the following:
\begin{corollary}
\label{cor:hom}
Let $A\in\RegRw$. Then $A\in \RegR$ if and only if $0\in \rho(A)$.
\end{corollary}
\begin{proof}
It remains to show that $A\in \RegR$ implies $0\in \rho(A)$ and this follows by Proposition \ref{prop:stability}.
\end{proof}

\begin{remark}
The assertion of Proposition \ref{prop:stability} does not hold if instead we only assume $A\in \RegJw$. Indeed, $-\Delta$ satisfies $\RegJw$ on $L^q(\R^d)$ with $q\in [2, \infty)$ (see \cite[Theorem 1.1 and Example 2.5]{MaximalLpregularity}), but of course $\omega_0(\Delta)=0$.
\end{remark}

\section{Independence of the time interval}
\label{s:independence}
\subsection{Independence of $T$}

It is well-known in deterministic theory of maximal $L^p$-regularity that
maximal regularity on a finite interval $J$ and exponential stability imply
maximal regularity on $\R_+$. We start with a simple result which allows to pass from $\R_+$ to any interval $(0,T)$.
\begin{proposition}\label{prop:RtoT}
Let $X$ be a UMD space with type $2$, let $p\in [2, \infty)$ and let $J=(0,T)$ with $T\in (0,\infty)$. If $A\in \RegR$, then $A\in \RegJ$.
\end{proposition}
\begin{proof}
Let $w>\omega_0(-A)$. Let $G\in L^p_{\F}(\O\times J;\g(H,X))$ and extending $G$ as $0$ on $(T,\infty)$ it follows that
\begin{align*}
\|S\diamond G\|_{L^p(\O\times J;D((w+A)^{1/2}))} & \leq \|S\diamond G\|_{L^p(\O\times\R_+;D((w+A)^{1/2}))}
\\ & \leq C \|G\|_{L^p(\O\times\R_+;\g(H,X))} = C \|G\|_{L^p(\O\times J;\g(H,X))}.
\end{align*}
\end{proof}

Next we present a stochastic version of \cite[Theorem 5.2]{Dore} of which its tedious proof is due to T.\ Kato. Our proof is a variation of the latter one.
\begin{theorem}\label{thm:TtoR}
Let $X$ be a UMD Banach space with type $2$ and let
$p\in [2,\infty)$. If $A\in \RegJ$ and $\omega_0(-A)<0$, then $A\in \RegR$.
\end{theorem}

\begin{proof}
It suffices to check the estimate in Proposition \ref{prop:omegaind}\eqref{it:omegaind2} with $w=0$. Let $J=(0,T)$ and for each $j\in \mathbb{N}$ set $T_j := jT/2$ and $G_j := \one_{[T_j,T_{j+1})} G$. In this proof, to shorten the notation below, we will write
\[\|G\|_{\gamma(a,b)}: = \|G\|_{\gamma((a,b);H,X)}.\]
It follows from the triangle inequality and \eqref{eq:Geq} that
\begin{align*}
\Big(\int_0^\infty  &\|s\mapsto A^{1/2}  S(t-s)G(s)\|_{\g(0,t)}^p \,
dt\Big)^{\frac1p} \\ & \leq \Big(\int_0^T \|s\mapsto A^{1/2} S(t-s)G(s)\|_{\g(0,t)}^p \, dt\Big)^{\frac1p}
\\ &  \quad + \Big(\sum_{j\geq 2} \int_{T_j}^{T_{j+1}}  \|s\mapsto  A^{1/2} S(t-s)G(s)\|_{\g(0,t)}^p \,
dt\Big)^{\frac1p}
\\ &  \leq \Big(\int_0^T  \|s\mapsto A^{1/2} S(t-s)G(s)\|_{\g(0,t)}^p \, dt\Big)^{\frac1p}
\\ &  \quad + \Big(\sum_{j\geq 2} \int_{T_j}^{T_{j+1}}  \|s\mapsto A^{1/2} S(t-s)G(s)\|_{\g(0,T_{j-1})}^p \, dt\Big)^{1/p}
\\ & \quad + \Big(\sum_{j\geq 2}\int_{T_j}^{T_{j+1}}   \|s\mapsto A^{1/2}  S(t-s)(G_{j-1}(s)+G_j(s))\|_{\g(T_{j-1},t)}^p \,
dt\Big)^{\frac1p}
\\ & =: R_1+R_2+R_3.
\end{align*}
By Proposition \ref{prop:omegaind}, to prove the claim, it is enough to estimate $R_i$ for $i=1,2,3$. By assumption, $A\in\RegJ$, then by Definition \ref{def:RegJ} one has
\begin{multline*}
R_1:=  \Big(\int_0^T  \|s\mapsto A^{1/2} S(t-s)G(s)\|_{\g(0,t)}^p \, dt\Big)^{\frac1p}
\leq C\|G\|_{L^p(J;X)}\leq C\|G\|_{L^p(\R_+;X)}.
\end{multline*}
Since $t-T/2\geq T_{j-1}$ for $t\in [T_{j},T_{j+1}]$, by \eqref{eq:phiGinfty} the second term can estimated as,
\[\begin{aligned}
R_2&=\Big(\sum_{j\geq 2} \int_{T_j}^{T_{j+1}} \|s\mapsto A^{1/2} S(t-s)G(s)\|_{\g(0,T_{j-1})}^p \, dt\Big)^{\frac1p} \\
& \leq \Big(\int_{T}^{\infty} \|s\mapsto A^{1/2} S(t-s)G(s)\|_{\g(0,t-\frac{T}{2})}^p \,dt\Big)^{\frac1p}.
\end{aligned}\]
By Theorem \ref{thm:analytic}, $(S(t))_{t\geq 0}$ is exponentially stable and analytic. Therefore, there are constants $a,M>0$ such that for all $t\in \R_+$ one has $\|A^{1/2}S(t)\|\leq M t^{-1/2} e^{-at/2}$.
By Proposition \ref{prop:inclusiongammatype}, for $t\geq T$ one has
\begin{align*}
\|s\mapsto & A^{1/2} S(t-s)G(s)\|_{\g(0,t-\frac{T}{2})}
\\& \leq \tau_{2,X} \|s\mapsto A^{1/2} S(t-s)G(s)\|_{L^2((0,t-\frac{T}{2});\g(H,X))}
\\ & \leq \tau_{2,X} \|s\mapsto M (t-s)^{-1/2} e^{-a (t-s)/2} G(s)\|_{L^2((0,t-\frac{T}{2});\g(H,X))}
\\ & \leq L \|s\mapsto e^{-a (t-s)/2} G(s)\|_{L^2((0,t-\frac{T}{2});\g(H,X))}
\\ & \leq L \Big(\int_0^t e^{-a (t-s)} \|G(s)\|_{\g(H,X)}^2 \, ds\Big)^{1/2}
\\ & = L (k*g)^{1/2},
\end{align*}
where $L = \tau_{2,X}  M (T/2)^{-1/2} $,  $k(s) = \one_{\R_+}(s) e^{-as}$ and $g(s) = \one_{\R_+}(s)\|G(s)\|_{\gamma(H,X)}^2$.
Taking $L^p(T,\infty)$-norms with respect to $t$, from Young's inequality we find that
\begin{align*}
R_2
& \leq L \|(k*g)^{1/2}\|_{L^p(\R)}
\leq L \|k\|_{L^1(\R)}^{1/2} \|g\|_{L^{p/2}(\R)}^{1/2}
= L a^{-1/2}\|G\|_{L^p(\R_+;\g(H,X))}.
\end{align*}

To estimate $R_3$, writing $G_{j-1, j} = G_{j-1}+G_j$ for each $j\geq 2$ we can estimate
\begin{align*}
R_{3j}^p& :=\int_{T_j}^{T_{j+1}} \|s\mapsto A^{1/2} S(t-s)G_{j-1, j}(s)\|_{\g(T_{j-1},t)}^p \, dt
\\ & = \int_{T_j}^{T_{j+1}} \|s\mapsto A^{1/2}S(t-s-T_{j-1})G_{j-1, j}(s+T_{j-1})\|_{\g(0,t-T_{j-1})}^p\, dt\\
& \leq \int_{T/2}^{T} \|s\mapsto A^{1/2} S(t-s)G_{j-1, j}(s+T_{j-1}))\|_{\g(0,t)}^p \,dt\\
& \leq \int_{0}^{T} \|s\mapsto A^{1/2} S(t-s)G_{j-1, j}(s+T_{j-1}))\|_{\g(0,t)}^p \,dt\\
& \leq C^p \|G_{j-1, j}(\cdot+T_{j-1})\|_{L^p(J;\g(H,X))}^p,
\end{align*}
where in the last step we have used the assumption and Proposition \ref{prop:omegaind}. Thus, for the third term we write
\[\begin{aligned}
R_3&=\Big(\sum_{j\geq 2} R_{3j}^p\Big)^{\frac1p}
\leq C\Big(\sum_{j\geq 2}\|G_{j-1, j}(\cdot+T_{j-1})\|_{L^p(J;\g(H,X))}^p\Big)^{\frac1p}
\\ & \leq 2C\Big(\sum_{j\geq 1}\|G_j\|_{L^p(\R_+;\g(H,X))}^p\Big)^{\frac1p}
 \leq 2C \|G\|_{L^p(\R_+;\g(H,X))}\,,
\end{aligned}\]
in the last step used that the $G_j$'s have disjoint support. This concludes the proof.
\end{proof}

Now we can extend Proposition \ref{prop:trans}.
\begin{corollary}\label{cor:T1T2}
Let $X$ be a UMD space with type $2$, let $p\in [2, \infty)$. Let $T_1<\infty$ and suppose that $A\in \RegJone$, then the following holds true:
\begin{enumerate}[{\rm (1)}]
\item For any $\lambda>\omega_0(-A)$ one has $\lambda+A\in \RegR$.
\item For any $T_2>0$,  $A\in \RegJtwointerval$.
\item If $T\in (0,\infty]$ and $\lambda>0$, then $\lambda A\in \RegJ$.
\end{enumerate}
\end{corollary}
\begin{proof}
(1): By Proposition \ref{prop:trans}\eqref{it:trans2} $\lambda+A\in \RegJone$ if $\lambda> \omega_0(-A)$. Since $\om_0(-(A+\lambda))<0$ for $\lambda>\omega_0(A)$, by Theorem \ref{thm:TtoR}, we obtain that $A+\lambda\in \RegR$.

(2): By (1) we know that there exists $w$ such that $A+w\in \RegR$. Now applying Proposition \ref{prop:RtoT} we find $w+A\in \RegJtwointerval$, and thus the result follows from Proposition \ref{prop:trans}\eqref{it:trans1}.

(3): Proposition \ref{prop:trans}\eqref{it:trans3} ensures that $\lambda A\in \RegJlambda$. Now (2) implies $\lambda A\in \RegJ$.
\end{proof}

\subsection{Counterexample}

In this final section we give an example of an analytic semigroup generator $-A$ such that $A\not\in\RegJ$.
\begin{proposition}\label{prop:counterex}
Let $X$ be an infinite dimensional Hilbert space. Then there exists an operator $A$ such that $-A$ generates an analytic semigroup with $\omega_0(-A)<0$, but $A\not\in \RegJ$ for any $T\in (0,\infty]$ and $p\in [2,\infty)$.
\end{proposition}
\begin{proof}
Let $(e_n)_{n\in \N}$ be a Schauder basis of $H$, for which there exists a $K>0$ such that for each finite sequence $(\alpha_n)_{n=1}^{N}\subset \C$ and
\begin{align*}
\Big\|\sum_{1\leq n\leq N} \alpha_n e_n \Big\|&\leq K\Big(\sum_{1\leq n\leq N} |\alpha_n|^2\Big)^{1/2},\\
\sup\Big\{\sum_{n\geq 1}|\alpha_n|^2\,&:\,\Big\|\sum_{n\geq 1} \alpha_n e_n\Big\|\leq 1\Big\}=\infty;
\end{align*}
for the existence of such basis see \cite[Example II.11.2]{Basis} and \cite[Example 10.2.32]{Analysis2}.
Then, define the diagonal operator $A$ by $A e_n = 2^n e_n$ with its natural domain. By \cite[Proposition 10.2.28]{Analysis2} $A$ is sectorial of angle zero and $0\in\rho(A)$. This implies that $-A$ generates an exponentially stable and analytic semigroup $S$ on $X$. In \cite[Theorem 5.5]{LeM} it was shown that for such operator $A$ there exists no $C>0$ such that for all $x\in D(A)$,
\begin{align*}
\|t\mapsto A^{1/2} S(t) x\|_{L^2(\R_+;X)}  \leq C\|x\|, \ \ \ x\in X.
\end{align*}
If $A	\in \RegR$, for some $p\in [2,\infty)$, then Lemma \ref{lem:gammaest} for $w=0$ provides such estimate (recall that for Hilbert space $X$ one has $\g(\R_+;X) = L^2(\R_+;X)$), this implies $A\notin \RegR$ for all $p\in [2, \infty)$. Since $\omega_0(-A)<0$, then Theorem \ref{thm:TtoR} shows that $A\notin\RegJ$ for any $T\in (0,\infty]$.
\end{proof}

\begin{remark}
The adjoint of the example in Proposition \ref{prop:counterex} gives an example of an operator which has $\RegRtwoo$, but which does not have a bounded $H^\infty$-calculus
(see \cite[Section 4.5.2]{ArendtHandbook}, \cite[Theorems 5.1-5.2]{LeM} and \cite[Example 10.2.32]{Analysis2}). Note that in the language of \cite{LeM} for the Weiss conjecture, $A\in \RegRtwoo$ if and only if $A^{1/2}$ is admissible for $A$. See \cite{LoVer} for more on this.
\end{remark}

\section{Perturbation theory}
\label{s:perturbation}

Combining the results of \cite{NVW11} (cf. Theorem \ref{thm:SMRmain}) with additive perturbation theory for the boundedness of the $H^\infty$-calculus, in many situations, one can obtain perturbation results for stochastic maximal regularity.
Perturbation theory for the boundedness of the $H^\infty$-calculus is quite well-understood. It allows to give conditions on $A$ and $B$ such that the sum $A+B$ has a bounded $H^\infty$-calculus again. Unfortunately, if $B$ is of the same order as $A$, then a smallness condition on $B$ is not enough (see \cite{McY}). Positive results can be found in \cite{DDHPV, KKW}.
In this section, we study more direct methods which give several other conditions on $A$ and $B$ such that the stochastic maximal regularity of $A$ implies stochastic maximal regularity of $\tilde{A}:=A+B$.

Fix $w>\omega_0(-A)$ and let $X_{\alpha} := D((w+A)^\alpha)$ with $\|x\|_{X_{\alpha}} = \|(w+A)^{\alpha} x\|$ for $\alpha>0$, and $X_{\alpha}$ is the completion of $X$ with $\|x\|_{X_{\alpha}} = \|(w+A)^{\alpha} x\|$ for $\alpha<0$ and $X_0:=X$.
These spaces do not dependent on the choice of $w$, and the corresponding norms for different values of $w$ are equivalent.
Moreover, for each $\beta,\alpha\in \R$, $(w+A)^{\alpha}:D((w+A)^{\alpha})\rightarrow R((w+A)^{\alpha})$ extends as to an isomorphism between $X_{\beta+\alpha}$ to $X_{\beta}$ and, with a slight abuse of notation, we will still denote the extension by $(w+A)^{\alpha}$. Lastly, define $A_{\alpha}:D(A_{\alpha})\subseteq X_{\alpha}\rightarrow X_{\alpha}$ where $D(A_{\alpha})=\{x\in X_{\alpha}\,:\,Ax\in X_{\alpha}\}$ the operator given by $A_{\alpha}x=Ax$ for $x\in D(A_{\alpha})$; see e.g. \cite{KKW,KuWePert} for more on this. Then if $-A$ generates a strongly continuous semigroup on $X$, then $-A_{\alpha}$ generates a strongly continuous semigroup $(S_{\alpha}(t))_{t\geq 0}$ on $X_{\alpha}$.

Lastly, in case $w+\tilde{A}$ is sectorial, consider the following condition for fixed $\alpha\in [1/2,1]$:
\begin{enumerate}
\item[(H)$_{\alpha}$]
$D((w+\tilde{A})^{\alpha})=X_{\alpha}$ and $D((w+\tilde{A})^{\alpha-\frac12})=X_{\alpha-\frac12}$.
\end{enumerate}

In Theorem \ref{thm:perturbationSMR}\eqref{it:perturbationSMR1} and \eqref{it:perturbationSMR2} below the smallness assumption already shows that $D(\tilde{A})=D(A)$. Therefore, in the important case $\alpha=1$ condition (H)$_{\alpha}$ reduces to the condition $D(\tilde{A}^{1/2})=D(A^{1/2})$.

The following is the main result of this section.
\begin{theorem}\label{thm:perturbationSMR}
Let $X$ be a UMD space with type $2$, let $p\in [2, \infty)$, $\alpha\in [1/2,1]$ and let $J = (0,T)$ with $T\in (0,\infty)$. Assume that $A\in \RegJ$, $B\in \calL(X_{\alpha},X_{\alpha-1})$ and set $\tilde{A} := (A_{\alpha-1}+B)|_{X}$. Then $\tilde{A}$ generates an analytic semigroup and $\tilde{A}\in \RegJ$ if $\emph{(H)}_{\alpha}$ holds and \textbf{at least one} of the following conditions is satisfied:
\begin{enumerate}[\rm (1)]
\item\label{it:perturbationSMR1} $A\in \MaxDetJ$. Moreover, for some $\varepsilon>0$ small enough, some $C>0$ and all $x\in X_{\alpha}$, one has
\[\|Bx\|_{X_{\alpha-1}}\leq \varepsilon\|x\|_{X_{\alpha}} +C \|x\|_{X_{\alpha-1}};\]
\item\label{it:perturbationSMR2} $B\in \calL(X_{\alpha},X_{\alpha-1+\delta})$ for some $\delta\in (0,1]$;
\item\label{it:perturbationSMR3} $-\tilde{A}$ generates a strongly continuous semigroup on $X$ and the operator $\tilde{A}_{\alpha-1}:=A_{\alpha-1}+B:X_{\alpha}\subset X_{\alpha-1}\to X_{\alpha-1}$ belongs to $\MaxDetJ$.
\end{enumerate}
\end{theorem}

Recall that $\MaxDetJ$ stands for deterministic maximal $L^p$-regularity. The result in \eqref{it:perturbationSMR1} is a relative perturbation result. In \eqref{it:perturbationSMR2} no deterministic maximal regularity is needed. The perturbation result in \eqref{it:perturbationSMR3} avoids an explicit smallness assumption of $B$ with respect to $A$. This result is inspired by \cite[Theorem 3.9]{VP18} where a more general setting is discussed in the case $\alpha=1$, but where a slightly different notion of stochastic maximal $L^p$-regularity is considered since there the spaces $X_{1/2}$ are assumed to be complex interpolation spaces (see \cite[Definition 3.5]{VP18}).

\begin{proof}[Proof of Theorem \ref{thm:perturbationSMR}\eqref{it:perturbationSMR1}]
{\em Step 1:} First we prove the result under the additional condition $C = 0$. This part of the argument is valid for $T\in (0,\infty]$. If $T=\infty$, then Proposition \ref{prop:stability} yields $\omega_0(-A)<\infty$. If $T<\infty$, then by Proposition \ref{prop:trans} we may assume $\omega_0(-A)<0$. It follows from \cite[Theorem 8, Remark 17]{KuWePert} that $-\tilde{A}$ generates an analytic semigroup; which we denote by $(\tilde{S}(t))_{t\geq 0}$. Moreover, for $\varepsilon$ small enough, we have $\om_0(-\tilde{A})<0$.
By Remark \ref{r:RegJdelta} and condition (H)$_{\alpha}$, we have to prove that there exists $C>0$ such for all for each $G\in L^p_{\F}(\Omega\times J;\g(H; X_{\alpha-1/2}))$,
\begin{equation}
\label{eq:pertC}
\|\tilde{S}\diamond G\|_{L^p(\Omega\times J;X_{\alpha})}\leq C \|G\|_{L^p(\Omega\times J;\g(H,X_{\alpha-1/2}))}.
\end{equation}
To do this, fix $G\in L^p_{\F}(\O\times J;\g(H,X_{\alpha-1/2}))$. Let us denote with $L$ the map from $L^p_{\F}(\O\times J;X_{\alpha})$ into itself given by
\begin{equation*}
L u = - S_{\alpha-1} * B u + S_{\alpha-1}\diamond G.
\end{equation*}
To see that $L$ maps $L^p_{\F}(\O\times J;X_{\alpha})$ into itself, note that $S\diamond G\in L^p_{\F}(\O\times J;X_{\alpha})$ since $A\in \RegJ$.
By assumption $A\in \MaxDetJ$ we also have $A_{\alpha-1}\in \MaxDetJ$. Thus
for $u,v\in L^p_{\F}(\O\times J;X_{\alpha})$,
\begin{align*}
\|L(u) - L(v)\|_{L^p(\O\times J;X_{\alpha})} &= \|S_{\alpha-1} * B (u - v)\|_{L^p(\O\times J;X_{\alpha})}  \\ & \leq C_{A,p} \|B(u-v)\|_{L^p(\O\times J;X_{\alpha-1})}
\\ & \leq C_{A,p} \varepsilon \|u-v\|_{L^p(\O\times J;X_{\alpha})}.
\end{align*}
Therefore, if $\varepsilon<1/C_{A,p}$, then $L$ is a strict contraction, and by Banach's theorem $L$ has a unique fixed point $u$. This yields
\begin{align}\label{eq:milduS}
u = - S_{\alpha-1} * B u + S\diamond G,
\end{align}
and
\begin{align*}
\|u\|_{L^p(\O\times J;X_{\alpha})} &= \|L(u)\|_{L^p(\O\times J;X_{\alpha})} \\ & \leq \|L(u)-L(0)\|_{L^p(\O\times J;X_{\alpha})} + \|L(0)\|_{L^p(\O\times J;X_{\alpha})}
\\ & \leq C_{A,p} \varepsilon\|u\|_{L^p(\O\times J;X_{\alpha})} + M\|G\|_{L^p_{\F}(\O\times J;\g(H,X_{\alpha-1/2}))}.
\end{align*}
Therefore,
\begin{equation}
\label{eq:per1reg}
\|u\|_{L^p(\O\times J;X_{\alpha})} \leq (1-C_{A,p} \varepsilon)^{-1} M \|G\|_{L^p_{\F}(\O\times J;\g(H,X_{\alpha-1/2}))}.
\end{equation}

To conclude, note that \eqref{eq:milduS} and ``mild solutions $\Rightarrow$ strong solutions'' (see Subsection \ref{ss:solutionconcepts}) implies that for all $t\in J$ a.s.
\[u(t) + \int_0^t A_{\alpha-1}u(s)+Bu(s) \,ds = \int_0^t G(s)\, dW_H(s).\]
Writing $A_{\alpha-1}u+Bu = \tilde{A} u$, ``strong solutions $\Rightarrow$ solutions mild'' yields that
\[u(t) = \tilde{S}\diamond G(t)\,, \qquad \forall t\in J.\]
This together with the inequality \eqref{eq:per1reg} concludes the proof of Step 1.

{\em Step 2:} Next assume $C>0$. We will show how one can reduce the proof to the case $C=0$. In this part of the proof we use $T<\infty$.
As before we can assume $\omega_0(-A)<0$ and $w=0$. Thus, $A$ is a sectorial operator and for each $s\in[0,1]$, the families of operators $\{A^{s}(\lambda+A)^{-s}: \lambda>0\}$ and $\{\lambda^{s}(\lambda+A)^{-s}: \lambda>0\}$ are uniformly bounded in $\calL(X)$ by a constant $M$ depending only on $A,w$ and $s$ (see \cite[Lemma 10, Remark 17]{KuWePert}).
The assumption can be rewritten as
\begin{equation}\label{eq:asrewr}
\|A^{\alpha-1} B A^{-\alpha}x\|\leq \varepsilon\|x\| +C \|{A}^{-1}x\|, \ \ \ x\in X.
\end{equation}
For each $\lambda>0$
and for $x\in X$, one has
\begin{align*}
\|(\lambda+A)^{\alpha-1} B (\lambda+A)^{-\alpha}x\|
& =\|(\lambda+A)^{\alpha-1} A^{1-\alpha} (A^{\alpha-1} B A^{-\alpha}) A^{\alpha} (\lambda+A)^{-\alpha}x\|\\
&\stackrel{(i)}{\lesssim_{A}} \|(A^{\alpha-1} B A^{-\alpha}) A^{\alpha} (\lambda+A)^{-\alpha}x\|\\
&\stackrel{(ii)}{\leq} \varepsilon\|A^{\alpha} (\lambda+A)^{-\alpha}x\|+C\|A^{-1+\alpha} (\lambda+A)^{-\alpha}x\|\\
&\stackrel{(iii)}{\lesssim_A} \varepsilon \|x\| +C \|(\lambda+A)^{-\alpha} x\|\\
& \stackrel{(i)}{\lesssim_A} \varepsilon \|x\| +C \lambda^{-\alpha}\|x\|,
\end{align*}
where in $(i)$ we used the uniform boundedness of $A^{\alpha}(\lambda+A)^{-\alpha}$ and $\lambda^{\alpha}(\lambda+A)^{-\alpha}$ for $\lambda>0$. In $(ii)$ we used \eqref{eq:asrewr}. In $(iii)$ we used that $0\in \rho(A)$ and $-1+\alpha\leq 0$.
If we choose $\varepsilon$ small enough and $\lambda>0$ large enough, then the condition of Step 1 holds, with the operator $A$ replaced by $A+\lambda$. Therefore, by Step 1 we obtain $\tilde{A}+\lambda$ generates an analytic semigroup and $\tilde{A}+\lambda\in \RegR$. Therefore, $\tilde{A}$ generates an analytic semigroup and Proposition \ref{prop:trans} implies that $\tilde{A}\in \RegJ$.
\end{proof}

If the perturbation is of a lower order, than the assumption that $A$ has deterministic maximal $L^p$-regularity can be avoided.
\begin{proof}[Proof of Theorem \ref{thm:perturbationSMR}\eqref{it:perturbationSMR2}]
As in the proof of \eqref{it:perturbationSMR1} one sees that $\tilde{A}$ generates an analytic semigroup. As in (1), due to Remark \ref{r:RegJdelta} and the hypothesis (H)$_{\alpha}$, we have only to show the estimate \eqref{eq:pertC}. Thanks to Corollary \ref{cor:T1T2}(2), we can prove the estimate \eqref{eq:pertC} where $J$ is replaced by any other interval $J_1:=(0,T_1)$, where $T_1$ will be chosen below.

Fix $G\in L^p_{\F}(\O\times J_1;\g(H,X_{\alpha-\frac12}))$. Let $L$ on $L^p_{\F}(\O\times J_1;X_{\alpha})$ be defined by
$L u = - S_{\alpha-1} * B u + S\diamond G$. By assumption we have $S\diamond G\in L^p_{\F}(\O\times J_1;X_{\alpha})$. Moreover,
by the analyticity of $S_{\alpha-1}$, for $u\in L^p_{\F}(\O\times J_1;X_{\alpha})$ we obtain
\begin{align*}
\|S_{\alpha-1} * B u(t)\|_{X_{\alpha}} &= \Big\|\int_0^t (w+A_{\alpha-1})^{1-\delta} S_{\alpha-1}(t-s) (w+A_{\alpha-1})^{\delta} B u(s) \, ds\Big\|_{X_{\alpha-1}}
\\ & \leq  C_{A,\delta} \|B\| \int_0^t (t-s)^{-(1-\delta)} \|u(s)\|_{X_{\alpha}} \, ds.
\end{align*}
Therefore, taking $L^p$-norms and Young's inequality yields
\begin{align*}
\|(w+A_{\alpha-1}) S_{\alpha-1} * B u(t)\|_{L^p_{\F}(\O\times J_1;X_{\alpha})}\leq C_{A,\delta} T^{\delta}_1  \|B\| \,\|u\|_{L^p_{\F}(\O\times J_1;X_{\alpha})}.
\end{align*}
Analogously for $u,v\in L^p_{\F}(\O\times J_1;X_{\alpha})$, one has
\begin{align*}
\|L(u) - L(v)\|_{L^p(\O\times J_1;X_{\alpha})} &= \|S_{\alpha-1} * B (u - v)\|_{L^p(\O\times J_1;X_{\alpha})} \\ & \leq C_{A,\delta} T_1^{\delta} \|B\| \, \|u-v\|_{L^p(\O\times J_1;X_{\alpha})}.
\end{align*}
Therefore, if $T_1$ is such that $C_{A,\delta} \|B\| T_1^{\delta}<1/2$, then $L$ is a contraction, and by Banach's theorem $L$ has a unique fixed point $u$. This yields
\begin{align}\label{eq:milduS2}
u =- S_{\alpha-1} * B u + S\diamond G,
\end{align}
and
\begin{align*}
\|u\|_{L^p(\O\times J_1;X_{\alpha})} &= \|L(u)\|_{L^p(\O\times J_1;X_{\alpha})} \\ & \leq \|L(u)-L(0)\|_{L^p(J\O\times J_1;X_{\alpha})} + \|L(0)\|_{L^p(\O\times J_1;X_{\alpha})}
\\ & \leq \frac12 \|u\|_{L^p(\O\times J_1;X_{\alpha})} + C_{A,\delta} T_1^{\delta}  \|B\| \, \|G\|_{L^p_{\F}(\O\times J_1;\g(H,X_{\alpha-1/2}))}.
\end{align*}
Therefore,
\[\|u\|_{L^p(\O\times J_1;X_{\alpha})} \leq 2C_{A,\delta} T_1^{\delta}  \|B\| \|G\|_{L^p_{\F}(\O\times J_1;\g(H,X_{\alpha-1/2}))}.\]
Now the proof can be completed as in the final part of Step 1 of the proof of \eqref{it:perturbationSMR1}.
\end{proof}

\begin{proof}[Proof of Theorem \ref{thm:perturbationSMR}\eqref{it:perturbationSMR3}]
This part of the proof also holds for $T=\infty$.

By assumption $-\tilde{A}$ generates a strongly continuous semigroup $\tilde{S}$ on $X$. Moreover, since $\tilde{A}_{\alpha-1}\in \MaxDetJ$, then $-\tilde{A}_{\alpha-1}$ generates an analytic semigroup $\tilde{S}_{\alpha-1}$ on $X_{\alpha-1}$;  see Subsection \ref{ss:DMR} or \cite[Corollary 4.2 and 4.4]{Dore}. Of course, if $\alpha=1$, then $\tilde{S}_{\alpha-1}=\tilde{S}$ and the first assumption is redundant.

By Proposition \ref{prop:trans} we may assume $\om_0(-\tilde{A})<0$, $J=\R_+$ and we set $w=0$. From here, the argument is the same performed in \cite[Theorem 3.9]{VP18} with minor modifications, so we only sketch the main step. To begin let $G\in L^p_{\G}(\O\times \R_+,w_{\alpha};\g(H,X_{\alpha-1/2}))$, since $A\in \RegR$, if $V:=S_{\alpha-1}\diamond G$ then
$$\|V\|_{ L^p(\O\times \R_+;X_{\alpha})}\lesssim_{\alpha,A} \|G\|_{L^p(\O\times \R_+;\g(H;X_{\alpha-1/2}))}.$$
Moreover, one can readily check that $U:=\tilde{S}\diamond G=V -\tilde{S}_{\alpha-1}*B V$, since $U$ is the unique weak solution to
$$
dU+\tilde{A} U dt=GdW_H, \qquad U(0)=0;
$$
cf. Subsection \ref{ss:solutionconcepts}. Since $\tilde{A}_{\alpha-1}\in \MaxDetR$, one has
\begin{align*}
\|U\|_{L^p(\Omega\times \R_+;D(\tilde{A}^{\alpha}))} & {\eqsim}_{\alpha, A, \tilde{A}} \|U\|_{L^p(\Omega\times \R_+;X_\alpha)}\\
&\leq \|S_{\alpha-1}\diamond G\|_{L^p(\Omega\times \R_+;X_{\alpha})} +\|\tilde{S}_{\alpha-1}*B V\|_{L^p(\Omega\times \R_+;X_{\alpha})}\\
&{\lesssim}_{A,\tilde{A},\alpha,p}\|S_{\alpha-1}\diamond G\|_{L^p(\Omega\times \R_+;X_{\alpha})} +\|B V\|_{L^p(\Omega\times \R_+;X_{\alpha-1})}\\
&\leq \|S_{\alpha-1}\diamond G\|_{L^p(\Omega\times \R_+;X_{\alpha})} +\|B\|\|V\|_{L^p(\Omega\times \R_+;X_{\alpha})}\\
&  {\lesssim}_{A,p,B} \|G\|_{L^p(\Omega\times \R_+;\g(H,X_{\alpha-1/2}))}\\
& {\eqsim}_{\alpha,A,\tilde{A},p,B} \| G\|_{L^p(\Omega\times \R_+;\g(H,D(\tilde{A}^{\alpha-1/2}))},
\end{align*}
where in the first and last step we have used (H)$_{\alpha}$. The conclusion follows by Remark \ref{r:RegJdelta} and Theorem \ref{thm:analytic}.
\end{proof}

\begin{remark}
Theorem \ref{thm:perturbationSMR}\eqref{it:perturbationSMR3} is also valid for $T=\infty$. If $C = 0$, then Theorem \ref{thm:perturbationSMR}\eqref{it:perturbationSMR1} also holds for $T=\infty$.
\end{remark}

\section{Weighted inequalities}
\label{s:weight}
\subsection{Preliminaries}\label{ss:preliminariesweights}
In this section we recall some basic fact about vector-valued Sobolev spaces and Bessel potential spaces with power weights. We refer to \cite{LMV18,MeyVer12} for details. Let $I\subseteq \R_+$ be an open interval and let $X$ be a Banach space. For $p\in (1,\infty)$, $\alpha\in \R$ and $w_{\alpha}(t):=t^{\alpha}$ we denote by $L^{p}(I,w_{\alpha};X)$ (or $L^p(a,b,w_{\alpha};X)$ if $I=(a,b)$) the set of all strongly measurable functions $f:I\rightarrow X$ such that
$$
\|f\|_{L^{p}(I,w_{\alpha};X)}:=\left(\int_I  \|f(t)\|_X^p w_{\alpha}(t)\,dt\right)^{1/p}<\infty.
$$
It is of interest to note that $w_{\alpha}$ belongs to the Muckenhoupt class $A_p$ if and only if $\alpha\in (-1,p-1)$. For $k\in \N$, let $W^{k,p}(I,w_{\alpha};X)$ denote the subspace of $L^p(I,w_{\alpha};X)$ of all functions for which $\partial^j f\in L^{p}(\R,w_{\alpha};X)$ for $j=0, \ldots, k$.

As usual, $\Sc(\R;X)$ denotes the space of $X$-valued Schwartz functions and $\Sc'(\R;X):=\calL(\Sc(\R);X)$ denotes the space of $X$-valued tempered distributions. Let $\mathcal{J}_s$ be the {\em Bessel potential operator} of order $s\in \R$, i.e.
\begin{equation*}
\mathcal{J}_s f=\mathcal{F}^{-1}((1+|\cdot|^2)^{s/2} \mathcal{F}(f))\,,\quad f\in \Sc(\R);
\end{equation*}
where $\mathcal{F}$ denotes the Fourier transform. Thus, one also has $\mathcal{J}_s:\Sc'(\R;X)\rightarrow \Sc'(\R;X)$. For $s\in \R$, $p\in (1,\infty)$, $\alpha\in(-1,p-1)$, $H^{s,p}(\R,w_{\alpha};X)\subseteq \Sc'(\R;X)$ denote the \textit{Bessel potential space}, i.e. the set of all $f\in \Sc'(\R;X)$ for which $\mathcal{J}_s f\in L^p(\R,w_{\alpha};X)$ and set $\|f\|_{H^{s,p}(\R,w_{\alpha};X)}:=\|\mathcal{J}_s f\|_{L^p(\R,w_{\alpha};X)}$.

To define vector valued weighted Bessel potential spaces on intervals, we use a standard method. Let $\Di(I;X) = C^\infty_c(I;X)$ with the usual topology and let $\Di'(I;X) = \calL(\Di(I), X)$ denote the $X$-valued distributions.
\begin{definition}
Let $p\in (1,\infty)$, $\alpha\in(-1,p-1)$ and $I\subseteq \R_+$ an open interval. Let
\begin{equation*}
H^{s,p}(I,w_{\alpha};X) = \{f\in \Di'(I;X)\,:\,\exists g \in H^{s,p}(\R,w_{\alpha};X)\\;\;\text{s.t.}\;\;g|_{I}=f\}\,,
\end{equation*}
endowed with the quotient norm $\|f\|_{H^{s,p}(I,w_{\alpha};X)}=\inf\{\|g\|_{H^{s,p}(\R,w_{\alpha};X)}\,:\,g|_{I}=f\}$.\\
Let $H^{s,p}_{0}(\R_+,w_{\alpha};X)$ be the closure of $C^{\infty}_c(\R_+;X)$ in $H^{s,p}(\R_+,w_{\alpha};X)$.
\end{definition}

To handle Bessel potential space on intervals we need the following standard result, which can be proved as in \cite[Propositions 5.5 and 5.6]{LMV18}, where the case $I = \R_+$ was treated.
\begin{proposition}
\label{prop:extension-interpolation}
Let $p\in (1,\infty)$, $\alpha\in (-1,p-1)$, and let $X$ be a UMD Banach space. Let $I\subseteq \R_+$ be an open interval.
\begin{enumerate}[\rm (1)]
\item For every $k\in \N$ there exists an extension operator $E_{k}:H^{s,p}(I,w_{\alpha};X)\rightarrow H^{s,p}(\R,w_{\alpha};X)$ such that $E_k f|_{I}=f$ for all $f\in H^{s,p}(I,w_{\alpha};X)$ and for each $s\in[0,k]$ and $E_k:C^k(\overline{I};X)\to C^k(\overline{I};X)$.
\item If $k\in \N$, $p\in (1, \infty)$, then $H^{k,p}(I,w_{\alpha};X) = W^{k,p}(I,w_{\alpha};X)$.
\item Let $\theta\in (0,1)$ and $s_0,s_1\in \R$ and set $s:=s_0(1-\theta)+\theta s_1$. Then
\begin{equation*}
[H^{s_0,p}(\R_+,w_{\alpha};X),H^{s_1,p}(\R_+,w_{\alpha};X)]_{\theta}=H^{s,p}(\R_+,w_{\alpha};X).
\end{equation*}
\end{enumerate}
\end{proposition}
In the case $I = (0,T)$ with $T\in (0,\infty)$ it is possible to construct $E_k$ such that its norm is $T$-independent (see \cite[Lemma 2.5]{MS12}).

The following density lemma will be used several times. Let $I$ denote an interval. We write $C^k_c(\overline{I};X)$ for the space of $X$-valued functions $f:\overline{I}\to X$ such that the derivatives up to order $k$ are continuous and bounded with compact support. Note that $C^k_c(\overline{I};X) = C^k(\overline{I};X)$ if $I$ is bounded.
\begin{lemma}\label{lem:densityinters}
Let $X$ and $Y$ be Banach spaces such that $Y\hookrightarrow X$ densely. Let $k\in \N$, $s\in [-k,k]$, $p\in (1, \infty)$, $\alpha\in (-1,p-1)$. Then $C^k_c(\overline{I})\otimes Y$ is dense in $H^{s}(I,w_{\alpha};X)$ and in $H^{s}(I,w_{\alpha};X)\cap L^p(I;w_{\alpha};Y)$.
\end{lemma}
\begin{proof}
By Proposition \ref{prop:extension-interpolation} it suffices to prove the statements in the case $I  =\R$. The density of $C^k_c(\R)\otimes X$ in $H^{s}(\R,w_{\alpha};X)$ follows from \cite[Lemma 3.4]{LMV18}. Now since $Y$ is densely embedded in $X$ the result follows.

To prove the density in $E:=H^{s}(\R,w_{\alpha};X)\cap L^p(\R;w_{\alpha};Y)$, let $f\in E$. Let $\varphi\in C^\infty_c(\R)$ be such that $\varphi\geq 0$ and $\|\varphi\|_1 = 1$. Let $\varphi_n(x) = n^{-1}\varphi(nx)$. Then $\varphi_n * f\to f$ in $E$. Therefore, it suffices to approximate $g = \varphi_n * f$ for fixed $n$. Since $g\in  H^{s,p}(\R,w_{\alpha};Y)$ and $H^{s,p}(\R,w_{\alpha};Y)\hookrightarrow E$ it suffices to approximate $g$ in $H^{s,p}(\R,w_{\alpha};Y)$. This follows from the first statement of the lemma.
\end{proof}

The following deep result follows from \cite[Proposition 6.6, Theorems 6.7 and 6.8]{LMV18}. The scalar unweighted case is due to \cite{Se}.
\begin{theorem}
\label{th:SobolevLMV}
Let $p\in (1,\infty)$, $\alpha\in (-1,p-1)$ and let $X$ be a UMD space. Then the following holds true:
\begin{enumerate}[{\rm (1)}]
\item If $k\in \N_0$ and $k+\frac{1+\alpha}{p}<s<k+1+\frac{1+\alpha}{p}$, then
\[H^{s,p}_{0}(\R_+,w_{\alpha};X) = \{f\in H^{s,p}(\R_+,w_{\alpha};X): \tr(f) = 0, \ldots, \tr(f^{(k)}) = 0\}.\]
\item
Let $\theta\in (0,1)$ and $s_0,s_1\in \R$, define $s:=s_0(1-\theta)+\theta s_1$. Suppose $s_0,s_1,s\not\in \N_0+(1+\alpha)/p$, then
\begin{equation*}
[H^{s_0,p}_{0}(\R_+,w_{\alpha};X),H^{s_1,p}_{0}(\R_+,w_{\alpha};X)]_{\theta}=H^{s,p}_{0}(\R_+,w_{\alpha};X).
\end{equation*}
\item The realization of $\partial_t$ on $L^p(\R_+,w_{\alpha};X)$ with domain $H^{1,p}_{0}(\R_+,w_{\alpha};X)$ has a bounded $H^{\infty}$-calculus of angle $\pi/2$. In particular, $D((\partial_t)^{s})= H^{s,p}_{0}(\R_+,w_{\alpha};X)$ provided $s\not\in \N_0+(1+\alpha)/p$.
\end{enumerate}
\end{theorem}

Let $A$ be a sectorial operator on a Banach spaces $X$ and assume $0\in \rho(A)$. As usual, for each $m\in \N$, we denote by $D(A^m)$ the domain of $A^m$ endowed with the norm $\|\cdot\|_{D(A^m)}:=\|A^m\cdot\|_X$. Then for each $\vartheta>0$ and $p\in (1,\infty)$ we define
$$D_A(\vartheta,p):=(X,D(A^m))_{\vartheta/m,p};$$
where $\vartheta<m\in \N$ and $(\cdot,\cdot)_{\vartheta/m,p}$ denotes the real interpolation functor (see e.g. \cite{BeLo,InterpolationLunardi,Tr1}). It follows from reiteration (see \cite[Theorem 1.15.2]{Tr1}) that $D_A(\mu,p)$ does not depend on the choice of $m>\vartheta$, moreover
$$
(X,D_A(\vartheta,p))_{\nu,q}=D_A(\nu\,\vartheta,q),
$$
for all $\nu>0$ and $q\in (1,\infty)$. We refer to \cite[Chapter 1]{Tr1}, \cite[Chapter 1]{InterpolationLunardi} and \cite[Chapter 3]{pruss2016moving} for more on this topic.

The following trace embedding is due to \cite[Theorem 1.1]{MV14} where the result was stated on the
full real line.  The result on $\R_+$ is immediate from the boundedness of the extension operator of Proposition \ref{prop:extension-interpolation} and the density Lemma \ref{lem:densityinters}.
\begin{theorem}
\label{th:traceMV}
Let $A$ be an invertible sectorial operator with dense domain and let $p\in(1,\infty)$, $\alpha\in (-1,p-1)$ and $k>s>(1+\alpha)/p$, where $k\in \N$. Then the trace operator $(\tr f):=f(0)$ initially defined on $C^k_c([0,\infty);D(A^k))$, extends to a bounded linear operator on $H^{s,p}(\R_+,w_{\alpha};X)\cap L^{p}(\R_+,w_{\alpha};D(A^s))$. Moreover,
\begin{equation*}
\tr: H^{s,p}(\R_+,w_{\alpha};X)\cap L^{p}(\R_+,w_{\alpha};D(A^s)) \rightarrow D_{A}(\mu,p),
\end{equation*}
where $\mu :=s-(1+\alpha)/p$.
\end{theorem}

The following proposition, besides its independent interest, will play an important role in the proof of Theorem \ref{th:regularitypath} below. There, for a Banach space $X$ and an interval $I$, $C_0(\overline{I};X)$ denotes the Banach space of all continuous functions on $\overline{I}$ with values in $X$ which vanish at infinity.
\begin{corollary}
\label{cor:trace}
Let $p\in (1,\infty)$, $\alpha\in [0,p-1)$ and let $X$ be a UMD space and define $I:=(0,T)\subseteq \R_+$ where $T\in (0,\infty]$. Let $A$ be an invertible sectorial operator on $X$. Then the following assertions hold:
\begin{enumerate}[{\rm (1)}]
\item[\rm(1)] If $s>(\alpha+1)/p$, then
\begin{equation*}
H^{s,p}(I,w_{\alpha};X)\cap L^p(I,w_{\alpha};D(A^{s}))\hookrightarrow C_0(\overline{I},D_A(\mu,p)),
\end{equation*}
where $\mu:=s-\frac{\alpha+1}{p}$.
\item[\rm(2)] If $s>1/p$ and $\delta>0$, then
\begin{equation*}
H^{s,p}(I_{\delta},w_{\alpha};X)\cap L^p(I_{\delta},w_{\alpha};D(A^{s}))\hookrightarrow  C_0(\overline{I}_{\delta};D_A(s-\tfrac{1}{p},p)),
\end{equation*}
where $I_{\delta}:=(\delta;T)$.
\end{enumerate}
\end{corollary}
By similar arguments as in \cite{MV14} using embedding theorems into Triebel--Lizorkin spaces one can avoid the use of the UMD property in the above result. We do not require this generality here and we only proof the special case.
\begin{proof}
By Proposition \ref{prop:extension-interpolation} it suffices to consider $I=\R_+$.

(1): To prove the required embedding by the density Lemma \ref{lem:densityinters} it suffices to check that $\sup_{t\geq 0}\|f(t)\|_{D_A(\mu,p)}\leq C\|f\|_E$ for every $f\in C^k_c(\overline{I};D(A^k))$, here $E:=H^{s,p}(\R_+,w_{\alpha};X)\cap L^p(\R_+,w_{\alpha};D(A^s))$. To prove this we extend a standard translation argument to the weighted setting. Let $(T(t))_{t\geq 0}$ the left-translation semigroup, i.e. $(T(t)f)(s):=f(t+s)$ on $L^p(\R_+;X)$. Since $\alpha\geq 0$, $T(t)$ is contractive on $L^p(\R_+,w_{\alpha};X)$ as well.
Since $T(t)$ commutes with the first derivative $\partial_s$ it is immediate that $(T(t))_{t\geq 0}$ defines a contraction on $W^{k,p}(\R_+,w_{\alpha};X)$. By complex interpolation and Proposition \ref{prop:extension-interpolation} it follows that there exists a constant $M$ such that $\|T(t)\|_{\calL(H^{s,p}(\R_+,w_{\alpha};X))}\leq M$ for $t\in \R_+$, and consequently the same holds on $E$. Now by Theorem \ref{th:traceMV} we obtain
\[\|f(t)\|_{D_A(\mu,p)} = \|(T(t)f)(0)\|_{D_A(\mu,p)}\leq C\|T(t)f\|_{E}\leq CM \|f\|_E.\]
as required.

(2): As before it suffices to estimate $\sup_{t\geq \delta}\|f(t)\|_{D_A(\mu,p)}$. Since $\alpha\geq 0$,
\[H^{s,p}(I_{\delta},w_{\alpha};X)\cap L^p(I_{\delta},w_{\alpha};D(A^{s}))\hookrightarrow H^{s,p}(I_{\delta};X)\cap L^p(I_{\delta};D(A^{s})).\]
Therefore, since (1) extends to any half line $[\delta, \infty)\subseteq [0,\infty)$ the required result follows from (1) in the unweighted case.
\end{proof}

\subsection{Weighted Stochastic Maximal $L^p$-regularity}
\label{ss:weightedSMR}
As before, in this section $X$ is a Banach space with UMD and type $2$.

For $p\in [2, \infty)$ and $\alpha\in \R$ and $T\in (0,\infty]$, let $L^p_{\F}(\O\times(0,T),w_{\alpha};X)$ denotes the closure of the adapted step processes in $L^p(\O; L^p((0,T),w_{\alpha};X)))$.

First we extend Definition \ref{def:RegJ} to the weighted setting:
\begin{definition}
Let $X$ be a UMD space with type 2, let $p\in [2,\infty)$, $w>\omega_0(-A)$,  $T\in (0,\infty]$ and $\alpha \in \R$. We say that $A$ belongs to $\RegJalpha$ if there is a constant $C$ such that for all $G\in L^p_{\F}(\O\times(0,T),w_{\alpha};\g(H,X))$ one has
\[
\|S\diamond G\|_{L^p(\O\times(0,T),w_{\alpha};D((w+A)^{1/2}))}\leq C \|G\|_{L^p_{\F}(\O\times(0,T),w_{\alpha};\g(H,X))}.
\]
\end{definition}

\begin{remark}
Note that for every $G\in L^p_{\F}(\O\times(0,T),w_{\alpha};\g(H,D(A)))$ the stochastic integral $(w+A)^{1/2}S\diamond G$ is well-defined in $X$. Indeed, since $\alpha<\frac{p}{2}-1$  by H\"older's inequality one obtains that for all $T<\infty$
\[L^p(0,T,w_{\alpha};X)\subseteq L^2(0,T;X);\]
and the claim follows as in Remark \ref{r:Stochasticintegralwelldef}.
\end{remark}

The main result of this subsection is a stochastic analogue of \cite[Theorem 2.4]{PruSim04}. \begin{theorem}
\label{th:independenceweight}
Let $X$ be a UMD space with type $2$, let $p\in [2, \infty)$ and $\alpha\in (-1,\frac{p}{2}-1)$. Then the following assertions are equivalent:
\begin{enumerate}[{\rm (1)}]
\item $A\in \RegR$.
\item $A\in \RegRalpha$.
\end{enumerate}
\end{theorem}
As a consequence $\RegRalpha = \RegR$ for all $\alpha\in (-1,\frac{p}{2}-1)$.

To prove the result we will prove the following more general result, which can be viewed as a stochastic operator-valued analogue of \cite{stein1957note}.
\begin{theorem}\label{thm:indweight}
Let $p\in [2, \infty)$, $\alpha\in (-\infty,\frac{p}{2}-1)$ and let $X$ be a Banach space and let $Y$ be a UMD Banach space with type $2$. Let $X_0$ be a Banach space which densely embeds into $X$. Let $\Delta = \{(t,s): 0<s<t<\infty\}$ and let $K\in C(\Delta;\calL(X,Y))$ be such that $\|K(t,s)\|\leq M/(t-s)^{1/2}$ and $\|K(t,s)x\|\leq M\|x\|_{X_0}$ for all $t>s>0$. For adapted step processes $G$ let $T_K G$ be defined by
\[T_K G(t)  = K\diamond G(t) = \int_0^t K(t,s) G(s) \, dW_H(s),  \ \  \ t\in \R_+.\]
Let $p\in [2, \infty)$ and $\alpha\in (-\infty,\frac{p}{2}-1)$. The following assertions are equivalent:
\begin{enumerate}[{\rm (1)}]
\item $T_K$ is bounded from $L^p_{\F}(\O\times\R_+,w_{\alpha};\g(H,X))$ into $L^p(\O\times\R_+,w_{\alpha};Y)$.
\item $T_K$ is bounded from $L^p_{\F}(\O \times\R_+;\g(H,X))$ into $L^p(\O \times\R_+;Y)$.
\end{enumerate}
\end{theorem}
As a consequence the boundedness of $T_K$ does not depend on $\alpha\in (-\infty,\frac{p}{2}-1)$.

To prove the theorem we prove a stochastic version of a standard lemma (see \cite{stein1957note}, \cite{Kree} and \cite[Proposition 2.3]{PruSim04}).
\begin{lemma}\label{lem:PruSim}
Let $X$ be a Banach space and let $Y$ be a UMD Banach space with type $2$. Let $p\in [2, \infty)$ and $\beta\in (-\infty,\frac{1}{2}-\frac{1}{p})$. Let $\Delta = \{(t,s): 0<s<t<\infty\}$. Let $K$ be as in Theorem \ref{th:independenceweight}. Then the operator $T_{K,\beta}:L^p_{\F}(\O\times\R_+;\g(H,X))\to L^p(\O\times\R_+;Y)$ defined by
\[T_{K,\beta} G(t)  = \int_0^t K(t,s) ( (t/s)^{\beta} -1) G(s) \, dW_H(s)\]
is bounded and satisfies $\|T_{K,\beta}\|\leq C_{p,Y} C_{\beta} M$.
\end{lemma}
\begin{proof}
By density it suffices to bound $T_{K,\beta} G$ for adapted step processes $G$.
Note that for all $t>s>0$ one has
\[\|K(t,s) ( (t/s)^{\beta} -1)\|^2 \leq M^2 k_{\beta}(t,s),\]
where $k_{\beta}:\{(s,t)\in (0,\infty)^2: s<t\}\to \R_+$ is given by
$k_{\beta}(t,s) =( (t/s)^{\beta} -1)^2/(t-s)$.

By Corollary \ref{cor:Ito} we have
\begin{align*}
\E\|T_{K,\beta} G(t)\|^p& \leq C_{p,Y}^p\E \Big(\int_0^t \|K(t,s) ( (t/s)^{\beta} -1)\|^2 \|G(s)\|^2_{\g(H,X)} \, ds\Big)^{p/2}
\\ & \leq C_{p,Y}^p M^p \E\Big(\int_0^t k_{\beta}(t,s) \|G(s)\|^2_{\g(H,X)} \, ds\Big)^{p/2},
\end{align*}
To conclude, it suffices to prove that
\[
\int_{\R_+} \Big(\int_0^t k_{\beta}(t,s) |f(s)|^2\, ds\Big)^{p/2} \, dt \leq C_{\beta}^p \|f\|_{L^p(\R_+)}^p\,,
\]
for any $f\in L^p(\R_+)$. Let us set $g(s) = |f(s) s^{1/p}|^2$ for $s>0$, then
\[\int_0^t k_{\beta}(t,s) |f(s)|^2\, ds= \frac{1}{t^{2/p}} \int_0^\infty h_{\beta}(t/s) g(s) \, \frac{ds}{s} = \frac{h_{\beta}*g(t)}{t^{2/p}},\]
where the convolution is in the multiplicative group $(*,\R_+\setminus\{0\})$ with Haar measure $d\mu(s) = \frac{ds}{s}$ and $h_{\beta}(x):= \one_{(1, \infty)}(x) \frac{(x^\beta-1)^2}{x-1} x^{2/p}$ for $x>0$. Taking $\frac{p}{2}$-powers and integrating over $t\in \R_+$ and applying Young's inequality yields
\begin{align*}
\int_{\R_+} \Big(\int_0^t k_{\beta}(t,s) |f(s)|^2\, ds\Big)^{p/2} \, dt   = \|h_{\beta} * g\|^{p/2}_{L^{p/2}(\R_+,\mu)} &\leq \|h_{\beta}\|^{p/2}_{L^1(\R_+,\mu)} \|g\|^{p/2}_{L^{p/2}(\R_+,\mu)}
\\ & = \|h_{\beta}\|^{p/2}_{L^1(\R_+,\mu)}  \|f\|_{L^p(\R_+)}^p.
\end{align*}
Finally, one easily checks that
\[\|h_{\beta}\|_{L^1(\R_+,\mu)} = \int_1^\infty \frac{(x^\beta-1)^2}{x-1} x^{2/p} \, \frac{dx}{x}\]
is finite if and only if $\beta<\frac{1}{2}-\frac{1}{p}$. This concludes the proof.
\end{proof}

\begin{proof}[Proof of Theorem \ref{thm:indweight}]
By density it suffices to prove uniform estimates for $T_K G$ where $G$ is a $X_0$-valued adapted step process.

$(1)\Rightarrow (2)$: \ Set $G_{\beta}(s):= s^{\beta} G(s)$ where $\beta=\alpha/p$. Observe that
\begin{equation}
\label{eq:weightequivalence1}
t^{\beta} T_K G(t) =  T_K G_{\beta}(t) + T_{K,\beta} G_{\beta}(t),
\end{equation}
where $T_{K,\beta}$ is as in Lemma \ref{lem:PruSim}. By (1) one has
\[\|T_KG_{\beta}\|_{L^p(\O\times\R_+;Y)}\leq C\|G_{\beta}\|_{L^p(\O\times\R_+;\g(H,X))}
=C\|G\|_{L^p(\O\times\R_+,w_{\alpha};\g(H,X))}.\]
Moreover, by Lemma \ref{lem:PruSim} one has
\[\|T_{K,\beta} G_{\beta}\|_{L^p(\O\times\R_+;Y)}
\leq C\|G_{\beta}\|_{L^p(\O\times\R_+;\g(H,X))} =C\|G\|_{L^p(\O\times\R_+,w_{\alpha};\g(H,X))}.\]
Then by \eqref{eq:weightequivalence1} and the previous estimates,
\begin{align*}
\|T_K G\|_{L^p(\O\times\R_+,w_{\alpha};Y)}&=\|t\mapsto t^{\beta}T_K G(t)\|_{L^p(\O\times\R_+;Y)}\\
&\leq \|T_KG_{\beta}\|_{L^p(\O\times\R_+;Y)}+\|T_{K,\beta} G_{\beta}\|_{L^p(\O\times\R_+;Y)}\\
&\leq 2C \|G\|_{L^p(\O\times\R_+,w_{\alpha};\g(H,X))}.
\end{align*}

$(2)\Rightarrow (1)$: \ Let $F_{-\beta}(s) = s^{-\beta} G(s)$ where $\beta=\alpha/p$. Similarly to \eqref{eq:weightequivalence1}, one has
\[T_K F(t) = t^{\beta} T_K F_{-\beta}(t) - T_{K,\beta} F(t).\]
As before, applying the assumption to $F_{-\beta}$ and Lemma \ref{lem:PruSim} gives that
\begin{align*}
\|T_K F\|_{L^p(\O\times\R_+;Y)}& \leq \|t\mapsto t^{\beta} T_K F_{-\beta}(t)\|_{L^p(\O\times\R_+;Y)} + \|T_{K,\beta}  F\|_{L^p(\O\times\R_+;Y)}\\
& =  \| T_K F_{-\beta}\|_{L^p(\O\times\R_+,w_{\alpha};Y)} + \|T_{K,\alpha}  F\|_{L^p(\O\times\R_+;Y)}
\\ & \leq C \|F_{-\beta}\|_{L^p(\O\times\R_+,w_{\alpha};\g(H,X))} + C''\|F\|_{L^p(\O\times\R_+;\g(H,X))}
\\ & = (C + C'')\|F\|_{L^p(\O\times\R_+;\g(H,X))},
\end{align*}
from which the result follows.
\end{proof}

\begin{proof}[Proof of Theorem \ref{th:independenceweight}]
If (1) holds, then by Theorem \ref{thm:analytic} the semigroup $S$ generated by $A$ is analytic.
To see that (2) also implies analyticity of $S$, note that the statement of Lemma \ref{lem:gammaest} still holds if instead we assume $A\in \RegRalpha$. To see this one can repeat the argument given there by using $\alpha>-1$. Therefore, if (2) holds, then Proposition \ref{prop:analyticgamma} implies that $S$ is analytic.

By the analyticity of $S$, the operator-valued family $K:\Delta \to \calL(X)$ defined by
\[K(t,s) := A^{\frac12} S(t-s)\]
satisfies $\|K(t,s)\|\leq C/(t-s)^{1/2}$ for $t>s>0$. Therefore, the equivalence of (1) and (2) follows from Theorem \ref{thm:indweight} with $X_0 = D(A)$.
\end{proof}

\subsection{Space-time regularity results}

To state the last results of this section, we introduce a further class of operators. From now on we will assume $(S(t))_{t\geq 0}$ is exponentially stable. For $\theta\in [0,1/2)$ we set
\begin{equation*}
S_{\theta}(t):=\frac{t^{-\theta}}{\Gamma(1-\theta)}S(t)\,, \qquad t \geq 0.
\end{equation*}

\begin{definition}
Let $X$ be a UMD space with type $2$, let $p\in [2, \infty)$, and $\theta\in [0,1/2)$ and assume $\omega_0(-A)<0$. We say that operator $A$ belongs to $\RegthetaR$ if for each $G\in L^p_{\F}(\O\times \R_+;\g(H,X))$ the stochastic convolution process
$$S_{\theta}\diamond G(t):=\int_0^t S_{\theta}(t-s) G(s)\,dW_H(s)\,,$$
 is well-defined in $X$, takes values in $D(A^{1/2-\theta})$ $\P\times dt$-a.e. and satisfies
\begin{equation*}
\|S_\theta\diamond G\|_{L^p(\O \times\R_+;D(A^{\frac12-\theta} ))}\leq C \|G\|_{L^p(\O \times\R_+;\g(H,X))}.
\end{equation*}
for some $C>0$ independent of $G$.
\end{definition}
By definition, we have $\normalfont{\text{SMR}}_0(p,\infty)=\RegR$.

The following important remark gives sufficient conditions for $A\in \RegthetaR$ which reduces to Theorem \ref{thm:SMRmain} if $\theta=0$.
\begin{remark}
\label{r:SMRthetafromHinfinite}
It was shown in \cite{MaximalLpregularity,NVW11,NVW13} that, if $X$ satisfies Assumption \ref{ass:JRbdd}, $0\in \rho(A)$ and $A$ has a bounded $H^{\infty}$-calculus of angle $<\pi/2$ then $A\in \RegthetaR$ for any $\theta\in [0,1/2)$ and $p\in (2,\infty)$. In addition, if $q=2$, then $A\in \RegthetaR$ for any $p\in [2,\infty)$. Lastly, the assumption $0\in \rho(A)$ can be avoided using a homogeneous version of $\RegthetaR$ (see \cite[Theorem 4.3]{MaximalLpregularity}).
\end{remark}
Before going further, we make the following observation:
\begin{proposition}
\label{prop:psi}
Let $X$ be a UMD space with type $2$ and let $p\in [2, \infty)$.
Let $A\in \RegthetaR$ be such that $\omega_0(-A)<0$ and $A$ is an $R$-sectorial operator of angle $\omega_R(A)<\pi/2$. Then, for any $0\leq \psi<\theta< 1/2$, we have $A\in\RegpsiR$.
\end{proposition}
\begin{proof}
First observe that an analogue of Proposition \ref{prop:omegaind} for $\RegthetaR$  holds and we will use it in the proof below.
By \cite[Lemma 3.3]{KKW} (or \cite[Proposition 10.3.2]{Analysis2}) the set $\{(sA)^{\theta-\psi}S(s/2)\,:\,s>0\}$ is $R$-bounded and hence $\g$-bounded
(see \cite[Theorem 8.1.3(2)]{Analysis2}). Therefore, by the $\g$-multiplier theorem (see \cite[Theorem 9.5.1]{Analysis2}) we obtain
\[\|s\mapsto A^{1/2-\psi}S_{\psi}(t-s)G(s)\|_{\g(0,t;H,X)}\\
\leq C \|s\mapsto A^{1/2-\theta}S_{\theta}((t-s)/2)G(s)\|_{\g(0,t;H,X)}.
\]
Taking $L^p$-norms on both sides we find that
\begin{align*}
\int_{0}^{\infty} \|s\mapsto A^{1/2-\psi} & S_{\psi}(t-s)G(s)\|_{\g(0,t;H,X)}^p\, dt
\\ & \leq C^p\int_{0}^{\infty} \|s\mapsto A^{1/2-\theta}S_{\theta}((t-s)/2)G(s)\|_{\g(0,t;H,X)}^p dt
\\ & =  \frac{C^p}{2} \int_{0}^{\infty} \|s\mapsto A^{1/2-\theta}S_{\theta}((2\tau-s)/2)G(s)\|_{\g(0,2\tau;H,X)}^p d\tau
\\ & \leq 2^{\frac{p}{2}-1} C^p \int_{0}^{\infty}  \|\sigma\mapsto A^{1/2-\theta}S_{\theta}(\tau-\sigma)G(2\sigma)\|_{\g(0,\tau;H,X)}^p d\tau
\\ & \leq  2^{\frac{p}{2}-1}  C^p K^p\|G\|_{L^{p}(\Omega\times\R_+;\g(H,X))}.
\end{align*}
where we only used elementary substitutions and in the last step we used the assumption applied to the function $G(2\cdot)$.
\end{proof}
The following proposition is the analogue of Theorem \ref{th:independenceweight} for the class $\RegthetaR$.
\begin{proposition}\label{prop:thetaresult}
Let $X$ be a UMD space with type $2$. Assume $\omega_0(-A)<0$ and $S$ is an analytic semigroup. Let $p\in [2, \infty)$, $\alpha\in (-1,\frac{p}{2}-1)$ and $\theta\in [0,1/2)$. Then the following are equivalent:
\begin{enumerate}[{\rm (1)}]
\item $A\in \RegthetaR$.
\item There is a constant $C>0$ such that for all $G\in L_\F^p(\O \times\R_+,w_{\alpha};\g(H,X))$ we have $S_\theta
\diamond G(t)\in D(A^{\frac12-\theta})$ $\P\times dt$-a.e. and
\begin{equation*}
\|S_\theta\diamond G\|_{L^p(\O \times\R_+,w_{\alpha};D(A^{\frac12-\theta}))}\leq C \|G\|_{L^p(\O \times\R_+,w_{\alpha};\g(H,X))}.
\end{equation*}
\end{enumerate}
\end{proposition}
\begin{proof}
Let $K_{\theta}:\Delta \to \calL(X)$ be defined by $K_{\theta}(t,s) = A^{\frac12-\theta}(t-s)^{-\theta}S(t-s)$.
By analyticity of the semigroup $(S(t))_{t\geq 0}$, one has $\|K_{\theta}(t,s)\|\leq C/(t-s)^{1/2}$ for $t>s>0$, and thus the result follows from Theorem \ref{thm:indweight} in the same way as in Theorem \ref{th:independenceweight}.
\end{proof}

We are ready to prove the main result of this section. Recall from Remark \ref{r:SMRthetafromHinfinite} that all the conditions are satisfied if $X$ is isomorphic to a closed subspace of $L^q$ with $q\in [2, \infty)$, $0\in \rho(A)$ and $A$ has a bounded $H^\infty$-calculus of angle $<\pi/2$.
\begin{theorem}
\label{th:regularitypath}
Let $X$ be a UMD space with type $2$. Assume $\omega_0(A)<0$, $A\in \BIP(X)$ with $\theta_A<\pi/2$. Let $p\in (2, \infty)$, let $\alpha\in (-1,\frac{p}{2}-1)$ (or $p=2$ and $\alpha=0$) and let $\theta\in [0,\frac12)$. Assume that $A\in \RegthetaR$.
\begin{enumerate}[{\rm (1)}]
\item[\rm(1)] {\rm (Space-time regularity)} If $\theta\neq (1+\alpha)/p$, then
\[
\E\|S\diamond G\|_{H^{\theta,p}(\R_+,w_{\alpha};D(A^{\frac12-\theta}))}^p \leq C^p
\, \E\|G\|_{L^p(\R_+,w_{\alpha};\g(H,X))}^p.
\]
\item[\rm(2)] {\rm (Maximal estimates)} If $\alpha\geq 0$ and $\theta-(1+\alpha)/p>0$, then
\[
\E \sup_{t\in\R_+} \| S\diamond G(t)\|_{D_A\left(\frac{1}{2}-\frac{1+\alpha}{p},p\right)}^p \leq
C^p \, \E\|G\|_{L^p(\R_+,w_{\alpha};\g(H,X))}^p.
\]
\item[\rm(3)] {\rm (Parabolic regularization)} If $\alpha\geq 0$ and $\theta-1/p>0$, then for any $\delta>0$
\[
\E \sup_{t\in[\delta,\infty)} \| S\diamond G(t)\|_{D_A\left(\frac{1}{2}-\frac{1}{p},p\right)}^p \leq
C^p \, \E\|G\|_{L^p(\R_+,w_{\alpha};\g(H,X))}^p.
\]
\end{enumerate}
In all cases the constant $C$ is independent of $G$.
\end{theorem}

\begin{proof}
To prepare the proof, we collect some useful facts.
Let $\mathscr{A}$ be the closed and densely defined
operator on $L^p(\R_+,w_{\alpha};X)$ with domain $D(\mathscr{A}) :=L^p(\R_+,w_{\alpha};D(A))$ defined by $$(\mathscr{A} f)(t) := A f(t);$$
since $A\in \BIP(X)$ then also $\mathscr{A}\in \BIP(L^p(\R_+,w_{\alpha};X))$ and $\theta_{\mathscr{A}}=\theta_{A}<\pi/2$. Moreover, $0\in \rho (\mathscr{A})$ since $0\in \rho(A)$.
Let $\mathscr{B}$ be the closed and densely defined operator on $L^p(\R_+,w_{\alpha};X)$
with domain $D(\mathscr{B}) := W^{1,p}_{0}(\R_+,w_{\alpha};X)$ given by
$$\mathscr{B} f := f'.$$
By Theorem \ref{th:SobolevLMV}, $\mathscr{B}$ has a bounded $H^{\infty}$-calculus of angle $\om_{H^\infty}(\mathscr{B})=\pi/2$; in particular $\theta_{\mathscr{B}}\leq \pi/2$.
Since $\theta_{\mathscr{A}}+\theta_{\mathscr{B}}<\pi$, by \cite[Theorems 4 and 5]{PrSo} the operator
\begin{align*}
\mathscr{C}:=\mathscr{A}+ \mathscr{B},  \quad D(\mathscr{C}) :=
D(\mathscr{A})\cap D(\mathscr{B}),
\end{align*}
is an invertible sectorial on $L^{p}(\R_+,w_{\alpha};X)$, moreover has bounded imaginary powers with $\theta_{\mathscr{C}}\leq \pi/2$.
By \cite[Proposition 3.1]{Brz2} one has
\begin{equation}\label{eq:Lambdainverse}
(\mathscr{C}^{-\gamma} f)(t)= \frac{1}{\Gamma(\gamma)}\int_0^t
(t-s)^{\gamma-1} S(t-s) f(s) \, ds.
\end{equation}
Moreover, for all $\gamma\in (0,1]$ one has
(see \cite[Lemma 9.5(b)]{PrussAnaltyicStefan})
\begin{equation}\label{eq:idLambdalambda}
\begin{aligned}
D(\mathscr{C}^{\gamma})&=[L^p(\R_+,w_{\alpha};X),D(\mathscr{B})]_{\gamma}\cap [L^{p}(\R_+,w_{\alpha};X),D(\mathscr{A})]_{\gamma}\\
&= H^{\gamma,p}_{0}(\R_+,w_{\alpha};X) \cap L^p(\R_+,w_{\alpha};D(A^{\gamma})),
\end{aligned}
\end{equation}
provided $\gamma\neq (1+\alpha)/p$, (the last equality follows from Theorem \ref{th:SobolevLMV}(2)).
To prove (1) and (2), by a density argument, it suffices to consider an
adapted rank step process $G:[0, \infty)\times\O\to \g(H,D(A))$.

(1): By the Da Prato--Kwapie\'n--Zabczyk factorization argument (see \cite{Brz2} and \cite[Section 5.3]{DPZ} and references therein), using \eqref{eq:Lambdainverse} for $\gamma=\theta$, the stochastic Fubini theorem and the
equality
\[\frac1{\Gamma(\theta)
\Gamma(1-\theta)}\int_r^t (t-s)^{\theta-1} (s-r)^{-\theta} \, ds = 1\]
one obtains, for all $t\in \R_+$,
\begin{equation}
\label{eq:identitytrace1}
\mathscr{C}^{-\theta} (A^{\frac12-\theta} S_\theta\diamond G)(t)
= A^{\frac12-\theta}S\diamond G(t) \ \ \text{almost surely}.
\end{equation}
Then,
\begin{align*}
\| A^{\frac12-\theta} S\diamond G\|_{L^p(\O;H^{\theta,p}(\R_+,w_{\alpha};X))}
& \stackrel{(i)}{\leq}C
\|\mathscr{C}^{\theta} A^{\frac12-\theta}S\diamond G\|_{L^p(\O \times\R_+,w_{\alpha};X)}
\\ & \stackrel{(ii)}{=} C
\|A^{\frac12-\theta}S_\theta\diamond G\|_{L^p(\O \times\R_+,w_{\alpha};X)} \\ &
\stackrel{(iii)}{\leq}  C'\|G\|_{L^p(\O \times\R_+,w_{\alpha};\g(H,X))},
\end{align*}
where in $(i)$ we have used \eqref{eq:idLambdalambda} (recall that by assumption $\theta\neq(1+\alpha)/p$), in $(ii)$ \eqref{eq:identitytrace1} and in $(iii)$ we used Proposition \ref{prop:thetaresult}.

(2): By Corollary \ref{cor:trace}(1), we have
\begin{equation*}
H^{\theta,p}(\R_+,w_{\alpha};X) \cap L^p(\R_+,w_{\alpha};D(A^{\theta}))
\hookrightarrow  C_0\left([0,\infty);D_A\left(\theta-\frac{1+\alpha}{p},p\right)\right).
\end{equation*}
Moreover, since $A\in \BIP(X)$ with $\theta_B<\pi/2$ then $\omega_{R}(A)<\pi/2$ thus $-A$ generates an analytic semigroup on $X$ (see Remark \ref{r:BIPimplies}). Setting $\zeta_{\lambda} = A^{\lambda}S\diamond G$, by Proposition \ref{prop:thetaresult} and the fact that $0\in\varrho(A)$, one has
\begin{equation}
\label{eq:thetatraceproof}
\begin{aligned}
\  \|\zeta_{\frac12-\theta}&\|_{L^p\left(\O;C_0\left([0,\infty);D_A\left(\theta-\frac{1+\alpha}{p},p\right)\right)\right)}
\\ & \leq K \|\zeta_{\frac12-\theta}\|_{L^p(\O;H^{\theta,p}([0,\infty),w_{\alpha};X))}+
K\|\zeta_{\frac12-\theta}\|_{L^p(\O;L^p(\R_+,w_{\alpha};D(A^{\theta})))}\\
& = K \|\zeta_{\frac12-\theta}\|_{L^p(\O;H^{\theta,p}(\R_+,w_{\alpha};X))}+K
\|\zeta_{\frac12}\|_{L^p(\O\times\R_+,w_{\alpha};X))}
\\ & \leq C K\|G\|_{L^p(\O \times\R_+,w_{\alpha};\g(H,X))}.
\end{aligned}
\end{equation}
Since $A^{\frac{1}{2}-\theta}:D_A(1/2-(1+\alpha)/p,p)\rightarrow D_A(\theta-(1+\alpha)/p,p)$ is an isomorphism (see \cite[Theorem 1.15.2 (e)]{Tr1}), we have
\begin{align*}
\|S\diamond  G\|_{L^p\left(\O;C_0\left([0,\infty);D_A\left(\frac{1}{2}-\frac{1+\alpha}{p},p\right)\right)\right)} &\eqsim_{A,\theta,p}
\|\zeta_{\frac12-\theta}\|_{L^p\left(\O;C_0\left([0,\infty);D_A\left(\theta-\frac{1+\alpha}{p},p\right)\right)\right)} \\ & \leq CK
\|G\|_{L^p(\O \times\R_+,w_{\alpha};\g(H,X))};
\end{align*}
where in the last inequality we have used \eqref{eq:thetatraceproof}.

(3): This follows from the same argument as in (2) using Corollary \ref{cor:trace}(2) instead of Corollary \ref{cor:trace}(1).
\end{proof}

\begin{remark}
\label{r:localization}
Similar to \cite[Remark 5.1]{MaximalLpregularity} (see also the references therein), Theorem \ref{th:regularitypath} can be localized via a standard stopping time argument. For future references, we give the explicit formulation for Theorem \ref{th:regularitypath}(3).\\
Let $\theta>1/p$, $0\leq \alpha<p/2-1$, $A\in \RegthetaR$ and let $\tau>0$ be an $\F$-stopping time then for any $G\in L^0_{\F}(\Omega;L^p((\tau,\infty),w_{\alpha};\g(H,X)))$,
$$
S\diamond G\in L^0\left(\Omega;C_0\left((\tau;\infty);D_A\left(\frac{1}{2}-\frac{1}{p},p\right)\right)\right).
$$
\end{remark}

\bibliographystyle{plain}
\bibliography{literature}

\def\polhk#1{\setbox0=\hbox{#1}{\ooalign{\hidewidth
  \lower1.5ex\hbox{`}\hidewidth\crcr\unhbox0}}} \def\cprime{$'$}
\begin{thebibliography}{10}

\bibitem{A18}
A.~Agresti.
\newblock {A quasilinear approach to fully nonlinear parabolic (S)PDEs on
  $\R^d$}.
\newblock {\em arXiv preprint arXiv:1802.06395}, 2018.

\bibitem{Am97}
H.~Amann.
\newblock Operator-valued {F}ourier multipliers, vector-valued {B}esov spaces,
  and applications.
\newblock {\em Math. Nachr.}, 186:5--56, 1997.

\bibitem{ArendtHandbook}
W.~Arendt.
\newblock Semigroups and evolution equations: functional calculus, regularity
  and kernel estimates.
\newblock In {\em Evolutionary equations. {V}ol. {I}}, Handb. Differ. Equ.,
  pages 1--85. North-Holland, Amsterdam, 2004.

\bibitem{BeLo}
J.~Bergh and J.~L{\"o}fstr{\"o}m.
\newblock {\em Interpolation spaces. {A}n introduction}.
\newblock Springer-Verlag, Berlin, 1976.
\newblock Grundlehren der Mathematischen Wissenschaften, No. 223.

\bibitem{BounitDrEl}
H.~Bounit, A.~Driouich, and O.~El-Mennaoui.
\newblock A direct approach to the {W}eiss conjecture for bounded analytic
  semigroups.
\newblock {\em Czechoslovak Math. J.}, 60(135)(2):527--539, 2010.

\bibitem{Brz2}
Z.~Brze{\'z}niak.
\newblock On stochastic convolution in {B}anach spaces and applications.
\newblock {\em Stochastics Stochastics Rep.}, 61(3-4):245--295, 1997.

\bibitem{ChillFio}
R.~Chill and A.~Fiorenza.
\newblock Singular integral operators with operator-valued kernels, and
  extrapolation of maximal regularity into rearrangement invariant {B}anach
  function spaces.
\newblock {\em J. Evol. Equ.}, 14(4-5):795--828, 2014.

\bibitem{ChillKrol}
R.~Chill and S.~Kr\'{o}l.
\newblock Weighted inequalities for singular integral operators on the
  half-line.
\newblock {\em Studia Math.}, 243(2):171--206, 2018.

\bibitem{DPZ}
G.~Da~Prato and J.~Zabczyk.
\newblock {\em Stochastic equations in infinite dimensions}, volume~44 of {\em
  Encyclopedia of Mathematics and its Applications}.
\newblock Cambridge University Press, Cambridge, 1992.

\bibitem{DDHPV}
R.~Denk, G.~Dore, M.~Hieber, J.~Pr{\"u}ss, and A.~Venni.
\newblock New thoughts on old results of {R}. {T}.\ {S}eeley.
\newblock {\em Math. Ann.}, 328(4):545--583, 2004.

\bibitem{DHP}
R.~Denk, M.~Hieber, and J.~Pr{\"u}ss.
\newblock {$R$}-boundedness, {F}ourier multipliers and problems of elliptic and
  parabolic type.
\newblock {\em Mem. Amer. Math. Soc.}, 166(788), 2003.

\bibitem{Dore}
G.~Dore.
\newblock Maximal regularity in {$L\sp p$} spaces for an abstract {C}auchy
  problem.
\newblock {\em Adv. Differential Equations}, 5(1-3):293--322, 2000.

\bibitem{EN}
K.-J. Engel and R.~Nagel.
\newblock {\em One-parameter semigroups for linear evolution equations}, volume
  194 of {\em Graduate Texts in Mathematics}.
\newblock Springer-Verlag, New York, 2000.

\bibitem{PrussAnaltyicStefan}
J.~Escher, J.~Pr\"{u}ss, and G.~Simonett.
\newblock Analytic solutions for a {S}tefan problem with {G}ibbs-{T}homson
  correction.
\newblock {\em J. Reine Angew. Math.}, 563:1--52, 2003.

\bibitem{HaNeVe}
B.~Haak, J.M.A.M. van Neerven, and M.C. Veraar.
\newblock A stochastic {D}atko-{P}azy theorem.
\newblock {\em J. Math. Anal. Appl.}, 329(2):1230--1239, 2007.

\bibitem{Haase:2}
M.H.A. Haase.
\newblock {\em The functional calculus for sectorial operators}, volume 169 of
  {\em Operator Theory: Advances and Applications}.
\newblock Birkh\"auser Verlag, Basel, 2006.

\bibitem{HornungDissertation}
L.~Hornung.
\newblock {\em {Semilinear and quasilinear stochastic evolution equations in
  Banach spaces}}.
\newblock PhD thesis, Karlsruher Institut für Technologie (KIT), 2017.

\bibitem{Hornung}
L.~Hornung.
\newblock Quasilinear parabolic stochastic evolution equations via maximal
  ${L}^p$-regularity.
\newblock {\em Potential Anal.}, pages 1--48, 2018.
\newblock Online first.

\bibitem{Analysis1}
T.P. Hyt\"onen, J.M.A.M.~van Neerven, M.C. Veraar, and L.~Weis.
\newblock {\em Analysis in {B}anach spaces. {V}ol. {I}. {M}artingales and
  {L}ittlewood-{P}aley theory}, volume~63 of {\em Ergebnisse der Mathematik und
  ihrer Grenzgebiete. 3. Folge.}
\newblock Springer, 2016.

\bibitem{Analysis2}
T.P. Hyt\"onen, J.M.A.M.~van Neerven, M.C. Veraar, and L.~Weis.
\newblock {\em Analysis in {B}anach spaces. {V}ol. {II}. {P}robabilistic
  {M}ethods and {O}perator {T}heory.}, volume~67 of {\em Ergebnisse der
  Mathematik und ihrer Grenzgebiete. 3. Folge.}
\newblock Springer, 2017.

\bibitem{KNVW}
N.~Kalton, J.M.A.M.~van Neerven, M.C. Veraar, and L.W. Weis.
\newblock Embedding vector-valued {B}esov spaces into spaces of
  {$\gamma$}-radonifying operators.
\newblock {\em Math. Nachr.}, 281(2):238--252, 2008.

\bibitem{Kalton-Kucherenko}
N.~J. Kalton and T.~Kucherenko.
\newblock Operators with an absolute functional calculus.
\newblock {\em Math. Ann.}, 346(2):259--306, 2010.

\bibitem{KKW}
N.J. Kalton, P.C. Kunstmann, and L.W. Weis.
\newblock Perturbation and interpolation theorems for the {$H\sp
  \infty$}-calculus with applications to differential operators.
\newblock {\em Math. Ann.}, 336(4):747--801, 2006.

\bibitem{PrussWeight1}
M.~K{\"o}hne, J.~Pr{\"u}ss, and M.~Wilke.
\newblock On quasilinear parabolic evolution equations in weighted
  {$L_p$}-spaces.
\newblock {\em J. Evol. Equ.}, 10(2):443--463, 2010.

\bibitem{Kree}
P.~Kr\'{e}e.
\newblock Sur les multiplicateurs dans {${\mathcal{F}} L^{p}$} avec poids.
\newblock {\em Ann. Inst. Fourier (Grenoble)}, 16:91--121, 1966.

\bibitem{Kry}
N.V. Krylov.
\newblock An analytic approach to {SPDE}s.
\newblock In {\em Stochastic partial differential equations: six perspectives},
  volume~64 of {\em Math. Surveys Monogr.}, pages 185--242. Amer. Math. Soc.,
  Providence, RI, 1999.

\bibitem{KryOverview}
N.V. Krylov.
\newblock A brief overview of the {$L_p$}-theory of {SPDE}s.
\newblock {\em Theory Stoch. Process.}, 14(2):71--78, 2008.

\bibitem{KuWePert}
P.~C. Kunstmann and L.W. Weis.
\newblock Perturbation theorems for maximal {${L}_p$}-regularity.
\newblock {\em Ann. Scuola Norm. Sup. Pisa Cl. Sci. (4)}, 30(2):415--435, 2001.

\bibitem{KuWe}
P.C. Kunstmann and L.W. Weis.
\newblock Maximal {$L\sb p$}-regularity for parabolic equations, {F}ourier
  multiplier theorems and {$H\sp \infty$}-functional calculus.
\newblock In {\em Functional analytic methods for evolution equations}, volume
  1855 of {\em Lecture Notes in Math.}, pages 65--311. Springer, Berlin, 2004.

\bibitem{LeM}
C.~Le~Merdy.
\newblock The {W}eiss conjecture for bounded analytic semigroups.
\newblock {\em J. London Math. Soc. (2)}, 67(3):715--738, 2003.

\bibitem{LMV18}
N.~Lindemulder, M.~Meyries, and M.~Veraar.
\newblock {Complex interpolation with Dirichlet boundary conditions on the half
  line}.
\newblock {\em Math. Nachr.}, 291(16):2435--2456.

\bibitem{LoVer}
E.~Lorist and M.C. Veraar.
\newblock {Singular stochastic integral operators}.
\newblock In preparation, 2018.

\bibitem{InterpolationLunardi}
A.~Lunardi.
\newblock {\em Interpolation theory}.
\newblock Appunti. Scuola Normale Superiore di Pisa (Nuova Serie). Edizioni
  della Normale, Pisa, second edition, 2009.

\bibitem{McY}
A.~McIntosh and A.~Yagi.
\newblock Operators of type {$\omega$} without a bounded {$H\sb \infty$}
  functional calculus.
\newblock In {\em Miniconference on {O}perators in {A}nalysis ({S}ydney,
  1989)}, volume~24 of {\em Proc. Centre Math. Anal. Austral. Nat. Univ.},
  pages 159--172. Austral. Nat. Univ., Canberra, 1990.

\bibitem{MS12}
M.~Meyries and R.~Schnaubelt.
\newblock Interpolation, embeddings and traces of anisotropic fractional
  {S}obolev spaces with temporal weights.
\newblock {\em J. Funct. Anal.}, 262(3):1200--1229, 2012.

\bibitem{MeyVer12}
M.~Meyries and M.C. Veraar.
\newblock Sharp embedding results for spaces of smooth functions with power
  weights.
\newblock {\em Studia Math.}, 208(3):257--293, 2012.

\bibitem{MV14}
M.~Meyries and M.C. Veraar.
\newblock Traces and embeddings of anisotropic function spaces.
\newblock {\em Math. Ann.}, 360(3-4):571--606, 2014.

\bibitem{NVW1}
J.M.A.M.~van Neerven, M.C. Veraar, and L.W. Weis.
\newblock Stochastic integration in {UMD} {B}anach spaces.
\newblock {\em Ann. Probab.}, 35(4):1438--1478, 2007.

\bibitem{NVW11eq}
J.M.A.M.~van Neerven, M.C. Veraar, and L.W. Weis.
\newblock Maximal {$L^p$}-regularity for stochastic evolution equations.
\newblock {\em SIAM J. Math. Anal.}, 44(3):1372--1414, 2012.

\bibitem{MaximalLpregularity}
J.M.A.M.~van Neerven, M.C. Veraar, and L.W. Weis.
\newblock Stochastic maximal {$L^p$}-regularity.
\newblock {\em Ann. Probab.}, 40(2):788--812, 2012.

\bibitem{NVW11}
J.M.A.M.~van Neerven, M.C. Veraar, and L.W. Weis.
\newblock On the {$R$}-boundedness of stochastic convolution operators.
\newblock {\em Positivity}, 19(2):355--384, 2015.

\bibitem{NVW13}
J.M.A.M.~van Neerven, M.C. Veraar, and L.W. Weis.
\newblock Stochastic integration in {B}anach spaces---a survey.
\newblock In {\em Stochastic analysis: a series of lectures}, volume~68 of {\em
  Progr. Probab.}, pages 297--332. Birkh\"{a}user/Springer, Basel, 2015.

\bibitem{NWa}
J.M.A.M.~van Neerven and L.W. Weis.
\newblock Invariant measures for the linear stochastic {C}auchy problem and
  {$R$}-boundedness of the resolvent.
\newblock {\em J. Evol. Equ.}, 6(2):205--228, 2006.

\bibitem{VP18}
P.~Portal and M.~Veraar.
\newblock Stochastic maximal regularity for rough time-dependent problems.
\newblock {\em arXiv preprint arXiv:1810.01183}, 2018.

\bibitem{PruSim04}
J.~Pr{\"u}ss and G.~Simonett.
\newblock {Maximal regularity for evolution equations in weighted
  $L_p$-spaces}.
\newblock {\em Archiv der Mathematik}, 82(5):415--431, 2004.

\bibitem{pruss2016moving}
J.~Pr{\"u}ss and G.~Simonett.
\newblock {\em Moving interfaces and quasilinear parabolic evolution
  equations}, volume 105.
\newblock Springer, 2016.

\bibitem{CriticalQuasilinear}
J.~Pr\"{u}ss, G.~Simonett, and M.~Wilke.
\newblock Critical spaces for quasilinear parabolic evolution equations and
  applications.
\newblock {\em J. Differential Equations}, 264(3):2028--2074, 2018.

\bibitem{PrSo}
J.~Pr{\"u}ss and H.~Sohr.
\newblock On operators with bounded imaginary powers in {B}anach spaces.
\newblock {\em Math. Z.}, 203(3):429--452, 1990.

\bibitem{RozVer17}
J.~Rozendaal and M.C. Veraar.
\newblock Fourier multiplier theorems on {B}esov spaces under type and cotype
  conditions.
\newblock {\em Banach J. Math. Anal.}, 11(4):713--743, 2017.

\bibitem{Se}
R.~Seeley.
\newblock Interpolation in {$L^{p}$} with boundary conditions.
\newblock {\em Studia Math.}, 44:47--60, 1972.

\bibitem{Basis}
I.~Singer.
\newblock {\em Bases in {B}anach spaces. {I}}.
\newblock Springer-Verlag, New York-Berlin, 1970.
\newblock Die Grundlehren der mathematischen Wissenschaften, Band 154.

\bibitem{stein1957note}
E.~M. Stein.
\newblock Note on singular integrals.
\newblock {\em Proceedings of the American Mathematical Society},
  8(2):250--254, 1957.

\bibitem{Tr1}
H.~Triebel.
\newblock {\em Interpolation theory, function spaces, differential operators}.
\newblock Johann Ambrosius Barth, Heidelberg, second edition, 1995.

\bibitem{VThesis}
M.C. Veraar.
\newblock {\em {Stochastic Integration in Banach spaces and Applications to
  Parabolic Evolution Equations}}.
\newblock PhD thesis, {Delft University of Technology}, 2006.
\newblock http://fa.its.tudelft.nl/~veraar/.

\bibitem{VerWeLPS}
M.C. Veraar and L.~Weis.
\newblock Estimates for vector-valued holomorphic functions and
  {L}ittlewood-{P}aley-{S}tein theory.
\newblock {\em Studia Math.}, 228(1):73--99, 2015.

\bibitem{We}
L.W. Weis.
\newblock Operator-valued {F}ourier multiplier theorems and maximal {$L\sb
  p$}-regularity.
\newblock {\em Math. Ann.}, 319(4):735--758, 2001.

\end{thebibliography}

\end{document}